%% file: main.tex
\documentclass[11pt]{article}
\usepackage[T1]{fontenc}
\usepackage{lmodern}
\usepackage{amsmath,amssymb,amsthm,mathtools}
\usepackage[margin=1in]{geometry}
\usepackage{graphicx,booktabs,enumitem,microtype}
\usepackage{placeins}
\usepackage{hyperref}
\hypersetup{hidelinks,pdftitle={Minimax Binary Counting under Shuffle Privacy},pdfauthor={Alex Shvets}}
\newtheorem{theorem}{Theorem}[section]
\newtheorem{proposition}[theorem]{Proposition}
\newtheorem{lemma}[theorem]{Lemma}
\newtheorem{corollary}[theorem]{Corollary}
\theoremstyle{definition}

\theoremstyle{remark}
\newtheorem{remark}[theorem]{Remark}
\newcommand{\Poi}{\operatorname{Poi}}
\newcommand{\Bin}{\operatorname{Bin}}
\newcommand{\TV}{\operatorname{TV}}
\newcommand{\Law}{\mathcal L}
\newcommand{\E}{\mathbb E}
\newcommand{\Prob}{\mathbb P}
\newcommand{\Sk}{\operatorname{Skellam}}

\newcommand{\eps}{\varepsilon}
\newcommand{\dd}{\mathrm d}
\DeclareMathOperator{\Var}{Var}

\newcommand{\Rstar}{R_*}
\newcommand{\Vfin}{V_{\mathrm{fin}}}
\newcommand{\lammin}{\lambda_{\min}}
\title{Minimax Binary Counting under Shuffle Privacy:\\
Poisson Limits and Randomized-Response Optimality}
\author{Alex Shvets\thanks{Independent researcher, Haifa, Israel.
ORCID: 0009-0005-9802-379X.}}
\date{October 2026}
\begin{document}
\maketitle
\begin{abstract}
We study estimation of a deterministic binary count from the shuffled
outputs of a common single-message local channel. Privacy is required for
every neighboring composition at fixed global parameters
$(\varepsilon,\delta)$, and estimators may be biased. We obtain uniform
critical Poisson and Skellam approximations and characterize the limiting
minimax risk as an infimum over finite observed-mark Poisson experiments.
Let $\Lambda$ be the noise intensity required by the boundary Poisson-shift
profile and let $\lambda_{\min}$ be the intensity required by the full
composition-maximized randomized-response profile. Binary randomized
response is asymptotically minimax optimal if and only if
$\Lambda=\lambda_{\min}$. When $\Lambda<\lambda_{\min}$, the minimax value
lies strictly between $(\Lambda+\lambda_{\min})/2$ and
$\lambda_{\min}$, and a randomized-response construction with observed
rare marks strictly improves on binary randomized response. We give a
finite-sample certified example. Binary randomized response is strictly
suboptimal at $\varepsilon=0$ for every $0<\delta<1$, and this persists
for sufficiently small positive $\varepsilon$ at each fixed target.
For every fixed positive $\varepsilon$, strict suboptimality
also occurs along a sequence of fixed privacy targets tending to zero.
The variational characterization permits approximation by finite
marked channels; it does not assert that a finite optimizer exists.
\end{abstract}

\input{intro}
\input{sections/tools}

\input{sections/uniform_rr}

\input{sections/sparse_uniform}

\input{sections/finite_risk}
\input{sections/rr_optimality}
\input{sections/strict_bounds}
\input{sections/variational}
\input{sections/certified_example}
\input{discussion}

\section*{Use of AI tools}
The author posed the problem, directed the research, and decided which
results to pursue and include. ChatGPT (OpenAI) assisted in developing and
checking proof strategies, implemented and ran the certificate scripts in
the ancillary files, performed numerical checks, and helped draft and
revise the text. Claude (Anthropic) was used for an independent audit of
the proofs, the numerical certificates, and the references. The author
checked all results and takes full responsibility for the content of this
paper.

\appendix
\input{sections/composition_appendix}

\input{sections/tail_appendix}
\input{sections/conditional_appendix}

\input{sections/plateau}

\input{references}
\end{document}

%% file: intro.tex
\section{Introduction}

In the single-message shuffle model, each user applies a local channel to
one private datum and an anonymizing shuffler releases the multiset of
messages. We consider a deterministic binary dataset and ask for the
smallest possible squared error in recovering its number of ones, subject
to a fixed global shuffled privacy constraint. The question requires both
mechanism design and statistical estimation: the channel may depend on the
population size, and the estimator need not be unbiased.

Binary randomized response (RR)~\cite{War65} supplies a natural benchmark.
If its local
crossover probability is $\tau_n$ with $n\tau_n\to\lambda$, only a bounded
number of messages are incorrect. The neighboring histogram experiment
then has a Poisson or Skellam limit. The privacy constraint, however,
requires a maximum over every neighboring composition, not just the
all-zero dataset and its one-one neighbor. We show that the discrepancy
between this boundary calibration and the global calibration determines
exactly whether binary RR is minimax optimal.

\subsection{Model and operational objective}

There are $n$ users with inputs $x_i\in\{0,1\}$. They independently apply
one common channel $W^{(n)}$ and send exactly one message each. Its output
alphabet $\mathcal Y_n$ is finite, with no restriction on its growth. The
analyst observes the full histogram $N$. Let $T_{n,k}$ be its distribution
when exactly $k$ inputs are one. Replacement adjacency gives neighboring
pairs $(T_{n,k},T_{n,k+1})$. Write $\Delta_n(\eps)$ for the tight,
two-sided privacy profile maximized over $0\le k<n$, defined explicitly
in~\eqref{eq:global-profile}. For a known channel let
\begin{equation}
 R_n(W^{(n)})=\inf_{\widehat k}\max_{0\le k\le n}
                  \E_{T_{n,k}}(\widehat k(N)-k)^2,
 \label{eq:count-risk}
\end{equation}
where the infimum is over all real-valued estimators. Channel randomization
is private and independent across users; correlated protocols, multiple
messages, and heterogeneous user-specific channels are outside the model.
At fixed $\eps\ge0$ and $0<\delta<1$ define
\begin{equation}
 \Rstar(\eps,\delta)=
 \inf_{\{W^{(n)}\}:\,\Delta_n(\eps)\le\delta\ \text{for all sufficiently large }n}
              \liminf_{n\to\infty}R_n(W^{(n)}).
 \label{eq:operational-risk}
\end{equation}
We call a sequence of channels \emph{admissible} at $(\eps,\delta)$
when it satisfies the eventual privacy constraint in this infimum.
This is the unnormalized error of the empirical count. It differs from
estimating an unknown population frequency from i.i.d. inputs. A Bernoulli
prior below is a tool for lower-bounding the deterministic-count risk.

\subsection{Main results}

For independent Poisson variables, let
\[
 D_{\lambda,\theta}=\Poi(\lambda(1-\theta))-\Poi(\lambda\theta),
 \qquad 0\le\theta\le1,
\]
and let $\Phi_{\lambda,\theta}(\eps)$ be the two-sided profile of its
unit-shift pair. Write $\Gamma_\lambda(\eps)=\max_\theta
\Phi_{\lambda,\theta}(\eps)$. The boundary profile is
$D_\eps(\lambda)=\Phi_{\lambda,0}(\eps)$. The unique level-$\delta$
root of $D_\eps$ is $\Lambda(\eps,\delta)$, and that of
$\Gamma_\lambda$ is $\lammin(\eps,\delta)$, with $\Lambda\le\lammin$.
The finite marked-channel objective $\Vfin$ is defined
in~\eqref{eq:finite-variational}; it retains the observed counts of every
rare mark in its privacy constraint.

\begin{theorem}[Operational minimax and the RR benchmark]
\label{thm:main}
For every fixed $\eps\ge0$ and $0<\delta<1$,
\[
 \Rstar(\eps,\delta)=\Vfin(\eps,\delta),\qquad
 \Rstar=\lammin\quad\Longleftrightarrow\quad\Lambda=\lammin.
\]
If $\Lambda<\lammin$, then
\[
 \frac{\Lambda+\lammin}{2}<\Rstar<\lammin.
\]
The infimum restricted to binary-output channels equals $\lammin$.
Whenever the unrestricted optimum is smaller, a symmetric four-output
channel strictly improves on binary RR.
\end{theorem}

Theorem~\ref{thm:main} combines Theorems~\ref{thm:variational},
\ref{thm:rr-classification}, \ref{thm:strict-midpoint}, and
\ref{thm:binary-optimality}. Its variational equality characterizes the
value but does not assert a finite optimizer or provide a closed-form
value throughout the strict-gap region.

The proof rests on three further results. First,
Theorems~\ref{thm:uniform-rr} and~\ref{thm:sparse-uniform} approximate the
neighboring experiments uniformly over composition, with explicit
finite-$n$ and intensity-remainder terms. Second,
Theorem~\ref{thm:finite-risk} identifies the exact asymptotic minimax risk
of every fixed finite marked channel, without a symmetry assumption.
Third, a paired-mark perturbation lowers the privacy profile at every
active global RR maximizer simultaneously, while reducing the limiting
count risk.
Proposition~\ref{prop:certified-example} certifies finite-sample privacy:
at $(\eps,\delta)=(0.001,0.46)$, a ternary channel has limiting risk $1$
and is private for $n\ge2500$, whereas $\lammin>1.015$.

Corollary~\ref{cor:small-target-gap} also proves that for every fixed
$\eps>0$, binary RR is strictly suboptimal along a sequence of fixed
targets $\delta_m\downarrow0$. This is an analytic existence result;
it gives no uniform size of the improvement.
Corollary~\ref{cor:near-zero-gap} proves strict RR suboptimality at
$\eps=0$ for every $0<\delta<1$, persisting on a positive-width
$\eps$ interval for each fixed target.

The appendices sharpen the RR benchmark. At $\eps=0$, the unique worst
composition is balanced, with $\Gamma_\lambda(0)=e^{-\lambda}I_0(\lambda)$.
For small positive $\eps$, the maximizers split into two interior branches.
For fixed $\eps>0$, the global large-noise profile is
$\Theta_\eps(\lambda^{-3/2}e^{-\lambda I(e^{-\eps})})$, where
$I(x)=x\log x-x+1$. These refinements do not replace the operational
all-channel comparison in Theorem~\ref{thm:main}.

\subsection{Related work}

The shuffle model and its amplification properties were developed
in~\cite{EFM19,CSU19,BBGN19,FMT21}. Numerical and mechanism-aware accounting
is studied in~\cite{KHH23,BJP22,Wang24,SCW25}. These methods address
finite-sample certification and more general channel classes; our
objective is the critical-regime count minimax under the exact global
constraint.

For Boolean sums, Cheu et al.~\cite[Theorem~2.2]{CSU19} give a single-bit
shuffled RR protocol with expected absolute error
$O(\eps^{-1}\sqrt{\log(1/\delta)})$ in their stated parameter range.
For summation over $[0,1]$ in the single-message shuffle model, Balle et al.~\cite[Theorems~4.1--4.2]{BBGN19}
establish the optimal worst-case MSE scale
$\Theta_{\eps,\delta}(n^{1/3})$ at fixed $0<\eps\le1$ and
$0<\delta<1/2$. The former is an order-level RR achievability result;
the latter concerns a strictly larger input domain. Neither identifies
the sharp fixed-parameter binary-count minimax value studied here.

Ghazi et al.~\cite[Appendix~D]{GKMP20} prove single-message binary-counting
lower bounds without a message-size restriction, using Bayes reduction,
row overlap, likelihood-ratio sets, and boundary testing; their bound is of
order $\log n$ for $\delta=n^{-c}$, with an unspecified constant. The
posterior mean reduction goes back to Balle et al.~\cite[Lemma~4.1]{BBGN19};
the Bernoulli identity in Lemma~\ref{lem:bayes-count} is its specialization.
Exact Poisson-shift calibration also appears in Appendix~F of the corrected
full version~\cite{GKMP20},
for their multi-message Poisson mechanism. Its standard centered count
estimator has error $\Poi(\lambda)-\lambda$, hence MSE $\lambda$
for every dataset. At $\lambda=\Lambda$ the mechanism is
$(\eps,\delta)$-private for every $n$.
Theorem~\ref{thm:boundary-converse} therefore shows that no common
single-message channel beats it asymptotically, and
Theorem~\ref{thm:strict-midpoint} makes the single-message loss strict in
the gap regime. We use these established tools and calibration objects,
and obtain the global critical minimax comparison and finite-channel
reduction above.

Takagi and Liew~\cite{TLMean26} study unbiased vector-mean estimation and
scalar RR in a high-privacy Gaussian setting under zero-out adjacency.
Their separate blanket-divergence analysis~\cite{TLBlanket26} supplies
shuffle indices and FFT accounting. Here the adjacency is replacement,
the shuffled privacy parameters are fixed, and the estimated quantity is
a deterministic binary count with unrestricted estimators. The improving
constructions below also admit unbiased count estimators; their advantage
comes from the critical global privacy constraint, rather than bias.
Corollary~\ref{cor:inversion} moreover gives
$\Rstar/\lammin\to1$ as $\delta\downarrow0$ at fixed $\eps>0$,
so strict finite-target separation need not persist at relative leading
order. Su et
al.~\cite{SCW25} give decomposition-based upper bounds valid for every
neighboring pair, with lower bounds from homogeneous-background pairs;
they leave the identification of the most vulnerable neighboring pair
open. Information-theoretic blanket-noise design is also
studied in~\cite{SCW26}.

Subsampling and variable communication are distinct design choices.
Girgis et al.~\cite{GDD21} analyze R\'enyi privacy with a fixed-size
sample chosen without replacement. The null-output channel below instead
erases each input independently with probability $c/n$, while retaining
the observable null count. Multi-message histogram protocols and the
pure-privacy equivalence of single-message shuffling and local privacy
are studied by Balcer and Cheu~\cite{BC20}. Our objective fixes
$0<\delta<1$ and requires exactly one output per user.

The fixed-channel Gaussian analysis~\cite{ShvGaussian} concerns a different
scaling regime and is not used as a black-box theorem here. The mechanisms
considered below are randomized-response channels augmented by a null
output or observed block labels. These augmentations are not claimed as
new channel families. Our contribution concerns their global
critical privacy behavior and the unrestricted minimax count-estimation
problem~\eqref{eq:operational-risk}.

\paragraph{Revision and dependencies.}
This version replaces arXiv:2603.10073v1 and is self-contained. The uniform RR and
sparse-channel limits are proved here as Theorems~\ref{thm:uniform-rr}
and~\ref{thm:sparse-uniform}. The preceding version's two-dominant hybrid
appendix is not reproduced, and its argument is withdrawn: its auxiliary
kernel used an inconsistent centering, and its unrestricted strong-boundary
privacy assertion discarded informative minority-block marks. None of
those assertions is needed here. This revision does not repair that
conditional-centering comparison and does not validate any hybrid
privacy-transfer argument based on it. All sparse rates retain their
parameter-convergence remainders. The present proof uses finite observed-mark
configurations; no arbitrary infinite-activity point-process representation
or finite-support optimizer theorem is assumed.

\paragraph{Organization.}
Sections~\ref{sec:tools}--\ref{sec:sparse} establish the privacy tools and
uniform experiments. Section~\ref{sec:finite-risk} proves the Bayes and
finite-channel risk formulas. Sections~\ref{sec:rr-calibration}--\ref{sec:strict-bounds}
give the calibrations, RR criterion, and universal bounds.
Section~\ref{sec:variational} proves the all-channel reduction, followed
by the certified example and the remaining open questions.

%% file: sections/tools.tex
\section{Privacy profiles and elementary tools}
\label{sec:tools}

\subsection{The shuffle experiment and the global profile}
Users hold $x_i\in\{0,1\}$ and independently release $Y_i$ through a common, possibly $n$-dependent, finite-output channel $W^{(n)}$. The transcript is the histogram $N(y)=\sum_{i=1}^n\mathbf1\{Y_i=y\}$. Let $T_{n,k}$ be its law for a dataset containing $k$ ones. Exchangeability implies that only $k$ matters. Neighboring pairs are $(T_{n,k},T_{n,k+1})$, $0\le k<n$.

For probability measures $P,Q$ on a common space and $\eps\ge0$, define
\[
\delta_{Q\Vert P}(\eps)=\sup_A\{Q(A)-e^\eps P(A)\},\qquad
\delta_{\rm two}(P,Q;\eps)=\max\{\delta_{Q\Vert P}(\eps),\delta_{P\Vert Q}(\eps)\}.
\]
The tight global privacy profile is
\begin{equation}
\Delta_n(\eps)=\max_{0\le k<n}\delta_{\rm two}(T_{n,k},T_{n,k+1};\eps).
\label{eq:global-profile}
\end{equation}
Thus the mechanism is $(\eps,\delta)$-DP exactly when $\Delta_n(\eps)\le\delta$. This is replacement adjacency, with fixed population size. We work with the convention $\TV(P,Q)=\sup_A|P(A)-Q(A)|=\frac12\|P-Q\|_1$.

\begin{lemma}[Privacy-profile identities and stability]
\label{lem:profile-tools}
On a countable space with masses $p,q$,
\[
\delta_{Q\Vert P}(\eps)=\sum_z(q(z)-e^\eps p(z))_+.
\]
If $Q\ll P$ with likelihood ratio $L$, this equals $\E_P[(L-e^\eps)_+]$. For pairs on a common space,
\begin{equation}
|\delta_{Q\Vert P}(\eps)-\delta_{Q'\Vert P'}(\eps)|
\le\TV(Q,Q')+e^\eps\TV(P,P').
\label{eq:profile-stability}
\end{equation}
The two-sided profiles differ by at most $(1+e^\eps)\max\{\TV(P,P'),\TV(Q,Q')\}$.
\end{lemma}
\begin{proof}
The maximizing set consists of points where $q-e^\eps p>0$; the density version follows by integration. For every $A$, the two objectives differ by at most the right side of~\eqref{eq:profile-stability}. Taking suprema in both directions proves the bound. Applying it to each directed profile and using the Lipschitz property of the maximum proves the last assertion.
\end{proof}

We use binary-experiment Le Cam distance only as a consequence of explicit TV or sufficient-statistic constructions. If $\mathcal E=(P_0,P_1)$ and $\mathcal F=(Q_0,Q_1)$, the deficiency is $\inf_K\max_{b=0,1}\TV(P_bK,Q_b)$ over Markov kernels, and the Le Cam distance is the maximum of the two deficiencies. On the same space, the identity kernel gives
\[
\Delta_{\rm LC}(\mathcal E,\mathcal F)\le\max_b\TV(P_b,Q_b).
\]
An invertible recentering gives exact equivalence of a selected binary pair. The composition-uniform bounds below allow the recentering to depend on the selected pair; they do not assert one global equivalence for a larger unknown-composition experiment.

\subsection{Poisson couplings}

\begin{lemma}[Rare counts and parameter perturbation]
\label{lem:poisson-tools}
For $0\le t\le1$, an integer $m\ge0$, and $\lambda,\lambda'\ge0$,
\begin{align}
\TV(\Bin(m,t),\Poi(mt))&\le mt(1-e^{-t})\le mt^2,\label{eq:bin-poi}\\
\TV(\Poi(\lambda),\Poi(\lambda'))&\le1-e^{-|\lambda-\lambda'|}\le|\lambda-\lambda'|.\label{eq:poi-perturb}
\end{align}
For a multinomial vector on a finite alphabet and a set $B$ of categories with total probability $t$, its counts restricted to $B$ differ from independent Poissons of the corresponding means by at most $mt(1-e^{-t})$ in TV.
\end{lemma}
\begin{proof}
For one summand, take $Z\sim\Poi(t)$. Set $X=1$ if $Z\ge1$, and if $Z=0$ set $X=1$ with probability $q=(t-1+e^{-t})/e^{-t}$. Then $0\le q\le1$ for $t\in[0,1]$, $X\sim\operatorname{Bern}(t)$, and
\[
\Prob(X\ne Z)=\Prob(Z\ge2)+e^{-t}q=t(1-e^{-t}).
\]
Independent copies and the union bound give~\eqref{eq:bin-poi}. To perturb a Poisson mean, add an independent Poisson with the difference of the means, after ordering them. The coupling differs with probability $1-e^{-|\lambda-\lambda'|}$.

For the multinomial statement, the total number in $B$ is $\Bin(m,t)$; for independent Poissons the total is $\Poi(mt)$. Given an equal total, both allocations have the same conditional multinomial law. Couple those allocations identically. If $t=0$, both vectors are zero; the argument also covers $t=1$.
\end{proof}

Independent product TV distances add as upper bounds: take independent couplings for the factors and use the union bound. Applying a measurable map cannot increase TV. These facts will be used for count differences and histogram maps.

\subsection{Log-concavity and unit shifts}

\begin{lemma}[Log-concavity under convolution]
\label{lem:log-concavity}
The convolution of two nonnegative summable sequences on $\mathbb Z$, each log-concave with no internal zeros, is log-concave with no internal zeros. In particular, a sum or difference of independent binomials or Poissons has a unimodal mass function. For any unimodal integer mass function $p$ vanishing at both tails,
\begin{equation}
\TV(p,p(\,\cdot-1))=\max_jp(j).
\label{eq:unimodal-tv}
\end{equation}
\end{lemma}
\begin{proof}
Here log-concavity means $u_j^2\ge u_{j-1}u_{j+1}$. With interval support, this implies nonnegativity of every $2\times2$ minor of the Toeplitz matrix $T_u(i,j)=u_{i-j}$: the relevant inequalities follow by multiplying the nonincreasing adjacent ratios. Conversely the adjacent minors give log-concavity. The identity $T_{u*v}=T_uT_v$ and the $2\times2$ Cauchy--Binet formula express each minor of the product as a sum of products of nonnegative minors. For finite truncations the sums are finite; truncation to intervals preserves log-concavity. Passing to the pointwise limit proves the claim. Convolution supports are intervals.

Binomial masses have decreasing adjacent ratios, as do Poisson masses on $\mathbb N$; reflection preserves log-concavity. This proves the particular cases. Finally, for a unimodal sequence the sum of its positive increments is its maximal value, and so is the sum of the absolute negative increments. Telescoping gives~\eqref{eq:unimodal-tv}, including a plateau of modes.
\end{proof}

%% file: sections/uniform_rr.tex
\section{A uniform critical approximation for randomized response}
\label{sec:uniform}

Binary randomized response flips each input with probability
$\tau_n\in(0,1/2)$ and has local privacy parameter
$\eps_0(n)=\log((1-\tau_n)/\tau_n)$. Write $K$ for the released number
of ones. Center a neighboring pair at its smaller composition $k$.
Among the $n-1$ common users,
\[
A\sim\Bin(n-1-k,\tau_n),\qquad B\sim\Bin(k,\tau_n)
\]
count false positives and false negatives, independently. Put $R_{n,k}=\Law(A-B)$ and let $R_{n,k}^+$ denote its unit shift. The switched user gives the exact representation
\begin{equation}
P_{n,k}=(1-\tau_n)R_{n,k}+\tau_nR_{n,k}^+,\qquad
Q_{n,k}=\tau_nR_{n,k}+(1-\tau_n)R_{n,k}^+.
\label{eq:exact-background}
\end{equation}
These are the laws of $K-k$ under $T_{n,k}$ and $T_{n,k+1}$. The same recentering is used under both hypotheses.

For $\lambda>0$, $\vartheta\in[0,1]$, define
\[
P_{\lambda,\vartheta}=\Law(D_{\lambda,\vartheta}),\qquad
Q_{\lambda,\vartheta}=\Law(1+D_{\lambda,\vartheta}),\qquad
\Phi_{\lambda,\vartheta}(\eps)=\delta_{\rm two}(P_{\lambda,\vartheta},Q_{\lambda,\vartheta};\eps).
\]
Let $p_{\lambda,\vartheta}(j)$ be the null mass. The directed profiles are exactly
\begin{align}
\delta_{Q_{\lambda,\vartheta}\Vert P_{\lambda,\vartheta}}(\eps)
&=\sum_{j\in\mathbb Z}(p_{\lambda,\vartheta}(j-1)-e^\eps p_{\lambda,\vartheta}(j))_+,\label{eq:skellam-forward}\\
\delta_{P_{\lambda,\vartheta}\Vert Q_{\lambda,\vartheta}}(\eps)
&=\sum_{j\in\mathbb Z}(p_{\lambda,\vartheta}(j)-e^\eps p_{\lambda,\vartheta}(j-1))_+.
\end{align}
Define the global critical envelope $\Gamma_\lambda(\eps)=\max_{\vartheta\in[0,1]}\Phi_{\lambda,\vartheta}(\eps)$; continuity below ensures that the maximum exists.

\begin{theorem}[Uniform binary-experiment approximation]
\label{thm:uniform-rr}
Fix $n\ge2$ and $\lambda>0$. Set $\lambda_n=(n-1)\tau_n$, $\vartheta_k=k/(n-1)$, and
\[
\beta_n=(n-1)\tau_n^2+\tau_n+|\lambda_n-\lambda|.
\]
Simultaneously for every $0\le k<n$,
\begin{equation}
\max\{\TV(P_{n,k},P_{\lambda,\vartheta_k}),
\TV(Q_{n,k},Q_{\lambda,\vartheta_k})\}\le\beta_n.
\label{eq:uniform-tv}
\end{equation}
The same bound holds for the Le Cam distance of the selected binary experiments. The term $(n-1)\tau_n^2$ may be replaced by $(n-1)\tau_n(1-e^{-\tau_n})$.
\end{theorem}
\begin{proof}
Lemma~\ref{lem:poisson-tools}, independent product coupling, and the difference map give
\[
\TV(R_{n,k},\Sk((1-\vartheta_k)\lambda_n,\vartheta_k\lambda_n))
\le(n-1)\tau_n(1-e^{-\tau_n}).
\]
Changing $\lambda_n$ to $\lambda$ costs at most
$|(1-\vartheta_k)(\lambda_n-\lambda)|+|\vartheta_k(\lambda_n-\lambda)|=|\lambda_n-\lambda|$.
Replacing the distinguished Bernoulli by $0$ under $P$ and $1$ under $Q$ costs at most $\tau_n$ in each coordinate of~\eqref{eq:exact-background}. Triangle inequalities give~\eqref{eq:uniform-tv}. The centered spaces coincide, so the identity kernels give the Le Cam bound.
\end{proof}

\begin{lemma}[Parameter continuity and symmetry]
\label{lem:parameter}
For fixed $\lambda$, both the null and shifted-null laws have TV distance at most
\[
1-e^{-\lambda|\vartheta-\vartheta'|}\le\lambda|\vartheta-\vartheta'|
\]
when $\vartheta$ is changed to $\vartheta'$. For fixed $\vartheta$, changing $\lambda$ costs at most $|\lambda-\lambda'|$. Consequently $\Phi_{\lambda,\vartheta}(\eps)$ is Lipschitz in these parameters, with constants $(1+e^\eps)\lambda$ in $\vartheta$ and $1+e^\eps$ in $\lambda$.
Moreover,
\begin{equation}
\Phi_{\lambda,1-\vartheta}(\eps)=\Phi_{\lambda,\vartheta}(\eps),\qquad
\delta_{Q_{\lambda,\vartheta}\Vert P_{\lambda,\vartheta}}(\eps)
=\delta_{P_{\lambda,1-\vartheta}\Vert Q_{\lambda,1-\vartheta}}(\eps).
\label{eq:reflection}
\end{equation}
\end{lemma}
\begin{proof}
Generate $N\sim\Poi(\lambda)$ jumps and independent uniform marks in $[0,1]$. A mark at most $\vartheta$ gives a negative jump and any other mark a positive jump. Changing $\vartheta$ changes no jump unless a mark falls between the two parameters. The count of such marks is $\Poi(\lambda|\vartheta-\vartheta'|)$, proving the first coupling bound. To change $\lambda$, add an independent marked Poisson process with the difference of the intensities. Apply Lemma~\ref{lem:profile-tools}.

The reflected noise $-D_{\lambda,\vartheta}$ has law $D_{\lambda,1-\vartheta}$. The common affine map $j\mapsto1-j$ transforms the ordered pair $(P_{\lambda,\vartheta},Q_{\lambda,\vartheta})$ to $(Q_{\lambda,1-\vartheta},P_{\lambda,1-\vartheta})$. This exchanges the directed profiles and preserves the two-sided profile.
\end{proof}

\begin{theorem}[Global critical privacy with an explicit error]
\label{thm:global-envelope}
In the setting of Theorem~\ref{thm:uniform-rr}, for every $\eps\ge0$,
\begin{equation}
\boxed{|\Delta_n(\eps)-\Gamma_\lambda(\eps)|
\le(1+e^\eps)\left(\beta_n+\frac{\lambda}{2(n-1)}\right).}
\label{eq:global-bound}
\end{equation}
In particular, if $n\tau_n\to\lambda\in(0,\infty)$, the global profiles converge, uniformly on every bounded $\eps$-interval. Under the calibration $e^{\eps_0(n)}=n/\lambda$, the error is $O(n^{-1})$ for fixed $\lambda$ and bounded $\eps$.
\end{theorem}
\begin{proof}
The profile of each finite-$n$ neighboring pair differs from $\Phi_{\lambda,\vartheta_k}$ by at most $(1+e^\eps)\beta_n$. This remains true after maximizing over $k$. Every point of $[0,1]$ is within $1/[2(n-1)]$ of the grid $\vartheta_k$. Lemma~\ref{lem:parameter} therefore bounds the difference between the grid maximum and the continuous maximum by $(1+e^\eps)\lambda/[2(n-1)]$.

If $n\tau_n\to\lambda$, then $\tau_n=O(n^{-1})$, $\beta_n\to0$. For $\tau_n=\lambda/(n+\lambda)$, each term of $\beta_n$ is $O(n^{-1})$. The common factor $1+e^\eps$ is bounded on a bounded interval.
\end{proof}

\subsection{Boundary limits and singular floors}
At $\vartheta=0$ the ordered pair is $(\Poi(\lambda),1+\Poi(\lambda))$. For $J\sim\Poi(\lambda)$ its likelihood ratio is $J/\lambda$, so
\begin{equation}
\delta_+(\eps)=\E[(J/\lambda-e^\eps)_+],\qquad
\delta_-(\eps)=\E[(1-e^\eps J/\lambda)_+].
\label{eq:poisson-profiles}
\end{equation}
Its reverse profile is at least $e^{-\lambda}$, and equals that floor once $e^\eps\ge\lambda$. At $\vartheta=1$ the canonical hypotheses are exchanged, as in~\eqref{eq:reflection}. In the interior, both Poisson means are positive and every integer has positive Skellam mass, so both directed profiles tend to zero as $\eps\to\infty$. The global envelope nevertheless satisfies $\Gamma_\lambda(\eps)\ge e^{-\lambda}$ because it includes the endpoints.

For finite $n$, the RR pair has common full support $\{0,\ldots,n\}$ before recentering and is $\eps_0(n)$-DP. Its two-sided profile therefore vanishes for $\eps\ge\eps_0(n)$. The limiting floor describes a bounded-central-$\eps$ asymptotic obstruction; it is not a contradiction to finite-$n$ local privacy or numerical amplification bounds.

%% file: sections/sparse_uniform.tex
\section{Sparse finite-output channels}
\label{sec:sparse}

We now keep the input alphabet binary and allow an arbitrary \emph{fixed finite output alphabet} $\mathcal Y$. Let $y_0\ne y_1$ be dominant outputs. Assume nonnegative finite intensities $\alpha_b(y)$, $y\ne y_b$, with
\[
nW_b^{(n)}(y)\to\alpha_b(y),\qquad b\in\{0,1\},\ y\ne y_b.
\]
The remaining probability is at $y_b$; it tends to one. Zero intensities are allowed, so rare probabilities need only be $O(1/n)$, not all $\Theta(1/n)$.

Define
\begin{align*}
A_b&=\sum_{y\ne y_b}\alpha_b(y),&L&=A_0+A_1,\\
q_{b,n}&=\sum_{y\ne y_b}W_b^{(n)}(y),&
\rho_b(n)&=\sum_{y\ne y_b}|nW_b^{(n)}(y)-\alpha_b(y)|,\\
E_n&=n\max_b\{q_{b,n}(1-e^{-q_{b,n}})\}+\max_b\rho_b(n).
\end{align*}
Let $e_y$ be the coordinate vector of output $y$. For $\vartheta\in[0,1]$ take independent
\[
U_y\sim\Poi((1-\vartheta)\alpha_0(y)),\quad y\ne y_0,
\qquad V_y\sim\Poi(\vartheta\alpha_1(y)),\quad y\ne y_1,
\]
and put
\[
H_\vartheta=\sum_{y\ne y_0}U_y(e_y-e_{y_0})+
\sum_{y\ne y_1}V_y(e_y-e_{y_1}),\qquad e=e_{y_1}-e_{y_0}.
\]
Let $\mathcal P_\vartheta=\Law(H_\vartheta)$ and $\mathcal Q_\vartheta=\Law(H_\vartheta+e)$, and write $\Psi_\vartheta$ for their two-sided profile. These are compound-Poisson histogram laws; the latent rare-category counts are independent Poissons, while their histogram coordinates need not be independent.

\begin{theorem}[Uniform sparse-channel approximation and global privacy]
\label{thm:sparse-uniform}
For all $0\le k<n$, center the transcript by $(n-k)e_{y_0}+ke_{y_1}$ under both hypotheses and call the resulting laws $\mathcal P_{n,k},\mathcal Q_{n,k}$. Then
\begin{align}
\TV(\mathcal P_{n,k},\mathcal P_{k/n})&\le E_n,\label{eq:sparse-p}\\
\TV(\mathcal Q_{n,k},\mathcal Q_{k/n})&\le E_n+L/n.\label{eq:sparse-q}
\end{align}
For the global finite-$n$ profile of this channel,
\begin{equation}
\boxed{\left|\Delta_n(\eps)-\max_{\vartheta\in[0,1]}\Psi_\vartheta(\eps)\right|
\le(1+e^\eps)(E_n+2L/n).}
\label{eq:sparse-global}
\end{equation}
In particular the profiles converge uniformly on bounded $\eps$-intervals. If $\rho_b(n)\le R/n$ for both $b$ and all sufficiently large $n$, then
\[
E_n\le\frac{(\max_bA_b+R)^2+R}{n}
\]
for those $n$, yielding an explicit $O(n^{-1})$ bound. The rate constant includes the quantitative remainder bound $R$.
\end{theorem}
\begin{proof}
Under composition $k$, the two input groups give independent multinomial vectors. Restrict each vector to its non-dominant outputs. Lemma~\ref{lem:poisson-tools} bounds their joint TV error to independent rare Poissons by
\[
(n-k)q_{0,n}(1-e^{-q_{0,n}})+kq_{1,n}(1-e^{-q_{1,n}})
\le n\max_b q_{b,n}(1-e^{-q_{b,n}}).
\]
The sum of errors in their Poisson means compared with those of $H_{k/n}$ is $(1-k/n)\rho_0(n)+(k/n)\rho_1(n)\le\max_b\rho_b(n)$. Applying the common linear histogram map gives~\eqref{eq:sparse-p}.

For the neighboring group sizes $n-k-1,k+1$, first compare to the means with proportions $1-k/n-1/n,k/n+1/n$. Their intensity-remainder error is again at most $\max_b\rho_b(n)$. Changing these limiting proportions back to $1-k/n,k/n$ costs at most $L/n$. The fixed common centering adds precisely $e$, giving~\eqref{eq:sparse-q}.

Poisson mean coupling gives $\TV(\mathcal P_\vartheta,\mathcal P_{\vartheta'})\le L|\vartheta-\vartheta'|$, and the same for shifted laws. Thus $\Psi_\vartheta$ is Lipschitz with constant $(1+e^\eps)L$. Every $\vartheta\in[0,1]$ is within $1/n$ of the grid $k/n$, $0\le k<n$, including its omitted upper endpoint. Profile stability and maximization give~\eqref{eq:sparse-global}.

Finally $nq_{b,n}\le A_b+\rho_b(n)\le A_b+R/n\le\max_bA_b+R$. Since $1-e^{-q}\le q$, the approximation term is at most $(\max_bA_b+R)^2/n$; the remainder term is at most $R/n$.
\end{proof}

%% file: sections/finite_risk.tex
\section{Exact Bayes identities and finite marked channels}
\label{sec:finite-risk}

\subsection{A Bayes identity for the full histogram}

The following identity applies to every finite output alphabet, including
an alphabet whose size depends arbitrarily on $n$. It concerns the full
histogram and therefore retains all information carried by message labels.
It specializes the posterior-mean reduction of
\cite[Lemma~4.1]{BBGN19}, also used for counting in
\cite[Appendix~D]{GKMP20}; the identity is recorded here as a proof tool.

\begin{lemma}[Exact Bayes risk for a shuffled count]
\label{lem:bayes-count}
Fix $n$, a finite-output channel $W$, and $0<\theta<1$. Under the prior
$K\sim\Bin(n,\theta)$, the minimum mean squared error for estimating $K$
from the histogram is
\begin{equation}
B_n(\theta)=n\theta(1-\theta)
\sum_{y\in\mathcal Y_n}
\frac{W_0(y)W_1(y)}{(1-\theta)W_0(y)+\theta W_1(y)}.
\label{eq:bayes-count}
\end{equation}
A summand with zero denominator is interpreted as zero. In particular,
$R_n\ge B_n(\theta)$. Writing
$c_n=\sum_y\min(W_0(y),W_1(y))$, we have
\begin{equation}
nc_n\le 4R_n.
\label{eq:overlap-bound}
\end{equation}
\end{lemma}

\begin{proof}
Take independent $X_i\sim\operatorname{Bern}(\theta)$ and generate the
messages independently through $W$. Conditional on the ordered messages,
the inputs are independent, with posterior success probabilities
\[
q_\theta(y)=\frac{\theta W_1(y)}{(1-\theta)W_0(y)+\theta W_1(y)}.
\]
Consequently their sum has conditional probability generating function
\[
\E[z^K\mid Y_1,\ldots,Y_n]
=\prod_{i=1}^n\bigl(1-q_\theta(Y_i)+q_\theta(Y_i)z\bigr)
=\prod_y\bigl(1-q_\theta(y)+q_\theta(y)z\bigr)^{N(y)}.
\]
This expression depends only on the histogram. It is therefore also the
conditional generating function given the histogram, and
\[
\Var(K\mid N)=\sum_yN(y)q_\theta(y)(1-q_\theta(y)).
\]
The posterior mean minimizes squared error. Taking the expectation of this
conditional variance and using
$\E N(y)=n((1-\theta)W_0(y)+\theta W_1(y))$ proves
\eqref{eq:bayes-count}.

Conditional on $K=k$, every input arrangement has the same histogram law
$T_{n,k}$. Thus this prior is a prior on the original count parameter, and
the maximum risk of any estimator is at least its average risk under this
prior. Taking the infimum over estimators yields $R_n\ge B_n(\theta)$.
Finally,
$uv/(u+v)\ge\min(u,v)/2$ for nonnegative $u,v$. Applying this inequality
at $\theta=1/2$ gives $B_n(1/2)\ge nc_n/4$, proving
\eqref{eq:overlap-bound}.
\end{proof}

\subsection{The finite marked experiment}

Let
\[
\mathfrak c=(a,b;(u_j,v_j)_{j=1}^J)
\]
be a fixed finite configuration of nonnegative intensities. For all
sufficiently large $n$, define a channel with dominant symbols $0,1$ and
distinct marks $z_j$, $j=1,\ldots,J$, by
\begin{align*}
W_0(1)&=a/n,& W_0(z_j)&=u_j/n,&
W_0(0)&=1-(a+\textstyle\sum_j u_j)/n,\\
W_1(0)&=b/n,& W_1(z_j)&=v_j/n,&
W_1(1)&=1-(b+\textstyle\sum_j v_j)/n.
\end{align*}
A mark with both rates zero can be discarded. Marks with just one positive
rate are allowed in this subsection. The observed mark counts are $N(z_j)$.

For $\theta\in[0,1]$, take mutually independent variables
\[
A\sim\Poi((1-\theta)a),\quad B\sim\Poi(\theta b),\quad
U_j\sim\Poi((1-\theta)u_j),\quad V_j\sim\Poi(\theta v_j),
\]
and put $M_j=U_j+V_j$, $Z=A-B-\sum_jV_j$. The limiting neighboring pair
observes the marks as well as the count coordinate:
\begin{equation}
\bigl(\Law(Z,M),\Law(Z+1,M)\bigr),\qquad
\Psi_{\mathfrak c,\theta}(\eps)
=\delta_{\rm two}\bigl(\Law(Z,M),\Law(Z+1,M);\eps\bigr).
\label{eq:marked-profile}
\end{equation}
The mark count $M_j$ has mean
$m_j(\theta)=(1-\theta)u_j+\theta v_j$. Given $M$, the variables $V_j$
are independent binomials with parameters $M_j$ and
\[
q_\theta(j)=\frac{\theta v_j}{m_j(\theta)}.
\]
When $u_j=0<v_j$, set $q_\theta(j)=1$ throughout $[0,1]$; when
$v_j=0<u_j$, set it to zero. At an endpoint where the denominator is zero,
$M_j=0$ almost surely, so these continuous conventions do not change the
experiment.

If $p_{\theta,M}$ denotes the conditional mass function of $Z$ given $M$,
then, with $h=e^\eps$,
\begin{align*}
\Psi_{\mathfrak c,\theta}(\eps)
=\max\bigg\{&\E_M\sum_{z\in\mathbb Z}
\bigl(p_{\theta,M}(z-1)-h p_{\theta,M}(z)\bigr)_+,\\
&\E_M\sum_{z\in\mathbb Z}
\bigl(p_{\theta,M}(z)-h p_{\theta,M}(z-1)\bigr)_+\bigg\}.
\end{align*}
The maximum is taken after averaging the two directions separately.
The pair is the recentered limiting histogram experiment, as shown in
Lemma~\ref{lem:marked-limit} below and in the conditional representation
of Proposition~\ref{prop:conditional}. This identification is used in
Sections~\ref{sec:rr-optimality} and~\ref{sec:variational}.

\begin{lemma}[Composition-uniform approximation for a finite configuration]
\label{lem:marked-limit}
For the channel associated with $\mathfrak c$,
\[
\Delta_n(\eps)\longrightarrow
\max_{\theta\in[0,1]}\Psi_{\mathfrak c,\theta}(\eps).
\]
The approximation error is $O_{\mathfrak c,\eps}(n^{-1})$. The map
$\theta\mapsto\Psi_{\mathfrak c,\theta}(\eps)$ is continuous.
\end{lemma}

\begin{proof}
Put $s_0=a+\sum_j u_j$, $s_1=b+\sum_jv_j$, and $S=s_0+s_1$.
Among the $n-k$ zero inputs, their counts of exceptional outputs $1,z_j$
are a multinomial rare-count vector. Among the $k$ one inputs, their
exceptional outputs are $0,z_j$. Coupling each vector with independent
Poisson counts of the same means gives total error at most
\[
\frac{(n-k)s_0^2+ks_1^2}{n^2}
\le\frac{s_0^2+s_1^2}{n}.
\]
The full histogram is determined by $(N(1),N(z_j):1\le j\le J)$.
After subtracting $k$ from its first coordinate, its Poisson approximation
is $(Z,M)$ at $\theta=k/n$. For the neighboring count $k+1$, the same
recentering gives $(Z+1,M)$ at $\theta'=(k+1)/n$. Changing $\theta'$ to
$\theta$ changes the sum of all latent Poisson means by at most $S/n$ in
absolute variation. An independent Poisson perturbation coupling bounds
the additional TV error by $S/n$. Stability of the privacy profile
therefore bounds the difference between the finite neighboring profile and
$\Psi_{\mathfrak c,k/n}$ by
\[
\frac{(1+h)(s_0^2+s_1^2+S)}{n}.
\]
The same latent Poisson coupling for two arbitrary compositions bounds
their observation laws in TV by $S|\theta-\theta'|$. Hence their two-sided
profiles differ by at most $(1+h)S|\theta-\theta'|$. This proves continuity
and allows the maximum over the grid $k/n$, $0\le k<n$, to be replaced by
the maximum over $[0,1]$ at an additional cost at most $(1+h)S/n$.
\end{proof}

\subsection{The exact asymptotic minimax risk}

Define the continuous function
\begin{equation}
B_{\mathfrak c}(\theta)=(1-\theta)a+\theta b+
\theta(1-\theta)\sum_{j=1}^J
\frac{u_jv_j}{(1-\theta)u_j+\theta v_j},\qquad 0\le\theta\le1.
\label{eq:finite-B}
\end{equation}
Terms with zero product are zero, including at endpoints. In particular,
$B_{\mathfrak c}(0)=a$ and $B_{\mathfrak c}(1)=b$.

\begin{theorem}[Risk of every fixed finite marked configuration]
\label{thm:finite-risk}
For the channel associated with a fixed finite configuration $\mathfrak c$,
\[
R_n\longrightarrow\max_{\theta\in[0,1]}B_{\mathfrak c}(\theta).
\]
No symmetry of the configuration and no unbiasedness requirement on the
estimator are imposed.
\end{theorem}

\begin{proof}
For every fixed $\theta\in(0,1)$, substituting the channel probabilities
into Lemma~\ref{lem:bayes-count} gives
$B_n(\theta)\to B_{\mathfrak c}(\theta)$. Indeed, the two dominant-symbol
terms converge to $(1-\theta)a$ and $\theta b$, and the mark terms are
exactly those in \eqref{eq:finite-B}. Thus
\[
\liminf_nR_n\ge\sup_{0<\theta<1}B_{\mathfrak c}(\theta)
=\max_{0\le\theta\le1}B_{\mathfrak c}(\theta),
\]
where continuity gives the last equality. This uses pointwise convergence
for fixed interior priors. It does not claim uniform convergence of the
Bayes risks at deterministic endpoint priors.

For the upper bound, consider first an oracle with the true
$\theta=k/n$. Give each message score $0$ on symbol $0$, $1$ on symbol $1$,
and $q_\theta(j)$ on mark $z_j$, and define
\[
\widehat k^*=N(1)-(1-\theta)a+\theta b+
\sum_jN(z_j)q_\theta(j).
\]
Put
\[
A_\theta=a+\sum_ju_jq_\theta(j),\qquad
C_\theta=b+\sum_jv_j(1-q_\theta(j)).
\]
The score mean is $A_\theta/n$ under input zero and $1-C_\theta/n$ under
input one. The identity
$(1-\theta)u_jq_\theta(j)=\theta v_j(1-q_\theta(j))$ shows that the oracle
is exactly unbiased. Summing the independent score variances gives
\begin{equation}
\E_k(\widehat k^*-k)^2
=B_{\mathfrak c}(\theta)
-\frac{(1-\theta)A_\theta^2+\theta C_\theta^2}{n}.
\label{eq:oracle-risk}
\end{equation}
For the mark contribution this calculation uses
\[
(1-\theta)u_jq_\theta(j)^2+
\theta v_j(1-q_\theta(j))^2
=\frac{\theta(1-\theta)u_jv_j}{(1-\theta)u_j+\theta v_j}.
\]
All constants in \eqref{eq:oracle-risk} are uniformly bounded on $[0,1]$.

Now replace the unknown composition by
$\widehat\theta=N(1)/n\in[0,1]$, obtaining the actual estimator
\[
\widehat k=N(1)-(1-\widehat\theta)a+\widehat\theta b+
\sum_jN(z_j)q_{\widehat\theta}(j).
\]
Let $T$ count all users whose output differs from their own dominant
symbol. Its summands are independent Bernoulli variables with probabilities
at most $C/n$, where $C=\max(s_0,s_1)$. For fixed $t>0$,
\[
\E_k e^{tT}\le\exp(C(e^t-1)),
\]
uniformly over $n$ and $k$. In particular its first four moments are
uniformly bounded. Deterministically,
$|N(1)-k|\le T$ and $\sum_jN(z_j)\le T$.

Each $q_\theta(j)$ is Lipschitz on $[0,1]$. If both rates are positive,
its derivative is $u_jv_j/((1-\theta)u_j+\theta v_j)^2$, bounded by
$\max(v_j/u_j,u_j/v_j)$; with one zero rate it is constant under our
convention. Let $L$ be a common Lipschitz constant. Then
\[
|\widehat k-\widehat k^*|
\le\frac{T}{n}\bigl(a+b+LT\bigr),
\qquad
\sup_k\|\widehat k-\widehat k^*\|_{L^2}=O(n^{-1}).
\]
The oracle errors have uniformly bounded second moments by
\eqref{eq:oracle-risk}. Expanding the squared error and applying
Cauchy--Schwarz therefore yields
\[
\sup_k\left|\E_k(\widehat k-k)^2-
\E_k(\widehat k^*-k)^2\right|=O(n^{-1}).
\]
The maximum risk of this estimator is at most
$\max_\theta B_{\mathfrak c}(\theta)+O(n^{-1})$, completing the upper
bound and the proof.
\end{proof}

Theorem~\ref{thm:finite-risk} concerns a fixed finite configuration.
Lemma~\ref{lem:bayes-count} and the overlap bound
\eqref{eq:overlap-bound}, in contrast, apply to arbitrary finite triangular
arrays and are the tools needed for the universal lower bounds below.

%% file: sections/rr_optimality.tex
\section{Calibrations, a boundary converse, and binary-output optimality}
\label{sec:rr-calibration}

Fix throughout this section $\eps\ge0$ and $0<\delta<1$, and write
$h=e^\eps$.  We extend the RR envelope to intensity zero by
$\Gamma_0(\eps)=1$, the profile of the pair of distinct point masses.
The endpoint Poisson profile is
\begin{equation}
 D_\eps(\lambda)
 =\delta_{\rm two}(\Law(J),\Law(1+J);\eps),
 \qquad J\sim\Poi(\lambda),\quad \lambda\ge0.
 \label{eq:boundary-profile}
\end{equation}
For $\lambda>0$, the likelihood ratio of the shifted law with respect to the
unshifted law is $j/\lambda$, including $j=0$.  Thus
\begin{equation}
 D_\eps(\lambda)=\max\left\{
 \E\left[(J/\lambda-h)_+\right],
 \E\left[(1-hJ/\lambda)_+\right]\right\},
 \qquad D_\eps(0)=1.
 \label{eq:boundary-profile-formula}
\end{equation}
We shall compare its calibration with the global RR calibration:
\begin{equation}
 \Lambda=\Lambda(\eps,\delta)
 :=\inf\{\lambda\ge0:D_\eps(\lambda)\le\delta\},
 \qquad
 \lammin=\lammin(\eps,\delta)
 :=\inf\{\lambda\ge0:\Gamma_\lambda(\eps)\le\delta\}.
 \label{eq:calibrations}
\end{equation}
All comparisons below concern the tight, two-sided global privacy profile.

\begin{theorem}[Well-defined calibrations]
\label{thm:calibration}
For every fixed $\eps\ge0$, both $D_\eps(\lambda)$ and
$\Gamma_\lambda(\eps)$ are continuous and strictly decreasing on
$[0,\infty)$, start at one, and tend to zero at infinity.  Consequently
$\Lambda$ and $\lammin$ are their unique positive roots at level $\delta$,
and
\begin{equation}
 D_\eps(\Lambda)=\Gamma_{\lammin}(\eps)=\delta,
 \qquad 0<\Lambda\le\lammin<\infty.
 \label{eq:calibration-roots}
\end{equation}
\end{theorem}

\begin{proof}
Poisson coupling and profile stability give continuity, with bound
$(1+h)|\lambda-\lambda'|$, for both functions.  For $\Gamma$, this bound
holds before and after taking the composition maximum.  Adding an
independent Poisson variable contracts each directed endpoint profile;
adding independent Skellam noise similarly contracts the profile at a fixed
composition.  We verify that these contractions are strict.

If $s$ is either signed mass sequence $p-hq$ or $q-hp$, write
$s=s_+-s_-$.  For any convolution kernel $K$,
\begin{equation}
 \sum_z(s(z))_+-\sum_z((K*s)(z))_+
 =\sum_z\min\{(K*s_+)(z),(K*s_-)(z)\}.
 \label{eq:strict-convolution-deficit}
\end{equation}
All sums converge absolutely.  At a Poisson endpoint with positive
intensity, each of the two signed sequences has both a positive and a
negative atom: for $p-hq$, use the atom at zero and a sufficiently large
index; for $q-hp$, use these in the opposite order.  A positive-intensity
Poisson increment has positive mass at every nonnegative integer.  At an
index above the chosen two atoms, both terms on the right side of
\eqref{eq:strict-convolution-deficit} are positive.  Hence the contraction
is strict.  Reflection treats the other endpoint.

For an interior composition, let $p_j$ be the Skellam mass with positive
Poisson means $a,b$.  Its recurrence
\[
 jp_j=a p_{j-1}-b p_{j+1}
\]
implies $p_j/p_{j-1}<a/j$ for $j\ge1$.  Applying the same inequality to
the reflected law gives $p_j/p_{j-1}\to\infty$ as $j\to-\infty$.
Thus each signed sequence has both signs.  A positive-intensity Skellam
increment at this composition has positive mass at every integer, so
\eqref{eq:strict-convolution-deficit} again has a positive term.  Each
directed profile, and therefore their maximum, strictly decreases in
$\lambda>0$.  Choose a composition attaining the maximum at the larger
intensity to obtain strict decrease of $\Gamma$.  At intensity zero the
profile is one, whereas at any positive intensity the shifted laws share
positive mass, so their two-sided profile is strictly below one.

Both functions tend to one at zero by coupling.  They tend to zero at
infinity uniformly in composition.  Indeed, Skellam masses are unimodal,
so their total variation distance from a unit shift is their maximal atom.
Fourier inversion gives, for every $\theta\in[0,1]$,
\[
 \Phi_{\lambda,\theta}(\eps)
 \le\Phi_{\lambda,\theta}(0)
 \le\frac1{2\pi}\int_{-\pi}^{\pi}e^{-\lambda(1-\cos u)}\,du
 \le\frac{\sqrt\pi}{2\sqrt{2\lambda}},
\]
using $1-\cos u\ge2u^2/\pi^2$ on that interval.  Finally
$\Gamma_\lambda(\eps)\ge D_\eps(\lambda)$ because the global envelope
includes the endpoints.  The root statements and their order follow.
\end{proof}

\subsection{RR achievability and the boundary lower bound}

For symmetric binary RR with crossover $\tau_n$ satisfying
$n\tau_n\to\lambda>0$, the estimator
\[
 \widehat k=\frac{N_1-n\tau_n}{1-2\tau_n}
\]
is unbiased at every dataset and has constant risk
$n\tau_n(1-\tau_n)/(1-2\tau_n)^2\to\lambda$.
Lemma~\ref{lem:bayes-count} at prior parameter $1/2$ gives the matching
lower bound $n\tau_n(1-\tau_n)\to\lambda$.  Thus the actual minimax
risk, allowing biased estimators, converges to $\lambda$.

For any $\lambda>\lammin$, strict calibration monotonicity gives
$\Gamma_\lambda(\eps)<\delta$.  Theorem~\ref{thm:global-envelope} then
makes RR with $\tau_n=\lambda/(n-1)$ exactly $(\eps,\delta)$-private for
all sufficiently large $n$.  Taking an infimum over these intensities yields
\begin{equation}
 \Rstar\le\lammin.
 \label{eq:rr-achievability}
\end{equation}
This upper bound can also be attained by a single sequence: choose
$\lambda_j\downarrow\lammin$ with $\lambda_j>\lammin$, and increasing
integers $N_j$ so that each fixed-$\lambda_j$ channel is private and its
risk is within $1/j$ of $\lambda_j$ whenever $n\ge N_j$.  On
$N_j\le n<N_{j+1}$ use that channel.  This diagonal sequence is eventually
private and its risk converges to $\lammin$.  In particular, achievability
does not require finite-$n$ privacy at the equality calibration itself.

\begin{theorem}[Universal boundary converse]
\label{thm:boundary-converse}
For every eventually $(\eps,\delta)$-private channel sequence in the model
of \eqref{eq:operational-risk}, with arbitrary finite output alphabets,
\[
 \liminf_{n\to\infty}R_n\ge\Lambda.
\]
Consequently $\Lambda\le\Rstar\le\lammin$.
\end{theorem}

\begin{proof}
Write $u_y=W_0^{(n)}(y)$, $v_y=W_1^{(n)}(y)$, and
$c_n=\sum_y\min(u_y,v_y)$.  The exact Bayes lower bound at $\theta=1/2$
gives
\begin{equation}
 R_n\ge\frac n2\sum_y\frac{u_yv_y}{u_y+v_y}
 \ge\frac{nc_n}{4}.
 \label{eq:overlap-risk}
\end{equation}
Zero denominators contribute zero.  If $r=\liminf R_n$ is infinite there
is nothing to prove.  Otherwise first choose a subsequence along which
$R_n\to r$, so $nc_n$ is bounded.

Choose deterministic $L_n\to\infty$ with $L_n/n\to0$, and define
\[
 A_n=\{y:v_y\ge L_nu_y\},\qquad
 a_n=nW_0^{(n)}(A_n),\qquad q_n=W_1^{(n)}(A_n).
\]
Simultaneous zero symbols are immaterial.  On $A_n$ the smaller row mass
is $u_y$, and on its complement $v_y\le L_n\min(u_y,v_y)$.  Therefore
\[
 a_n\le nc_n\le4R_n,
 \qquad 1-q_n\le L_nc_n\le4L_nR_n/n.
\]
Pass to a further subsequence with $a_n\to a<\infty$; then $q_n\to1$.
For each fixed $0<\theta<1$, the contributions of $A_n$ to
Lemma~\ref{lem:bayes-count} give
\[
 R_n\ge
 \frac{(1-\theta)a_n}{1+(1-\theta)/(\theta L_n)}.
\]
Take $n\to\infty$ on this same subsequence and then $\theta\downarrow0$
to obtain $r\ge a$.  This step is pointwise in $\theta$ and does not
require uniform Bayes-risk convergence at the prior endpoints.

Now postprocess the histogram to the count $N(A_n)$ and consider the
all-zero dataset and its one-one neighbor.  Their exact laws are
\[
 \Bin(n,a_n/n),\qquad
 \Bin(n-1,a_n/n)+\operatorname{Ber}(q_n).
\]
They converge in total variation to $\Poi(a)$ and $1+\Poi(a)$.
For example, their respective TV errors are bounded by
\[
 a_n^2/n+|a_n-a|,
 \qquad a_n^2/n+a_n/n+1-q_n+|a_n-a|.
\]
Data processing, eventual privacy, and profile stability imply
$D_\eps(a)\le\delta$.  Since $D_\eps(0)=1>\delta$, calibration yields
$a\ge\Lambda$, proving $r\ge\Lambda$.  The proof is independent of
the sizes of the finite alphabets.
\end{proof}

\subsection{An asymmetric comparison and all binary-output channels}

For $a,b>0$, let $G_{a,b}(\eps)$ be the global profile of unit-shift pairs
whose unshifted noise is
\[
 \Poi(a(1-\theta))-\Poi(b\theta),\qquad 0\le\theta\le1.
\]
These are the limits of binary channels with crossovers $a/n$ and $b/n$.

\begin{lemma}[Harmonic comparison]
\label{lem:harmonic}
With $H=2ab/(a+b)$,
\begin{equation}
 G_{a,b}(\eps)\ge\Gamma_H(\eps).
 \label{eq:harmonic-comparison}
\end{equation}
In particular, $G_{a,b}(\eps)\le\delta$ implies $H\ge\lammin$.
\end{lemma}

\begin{proof}
Reparameterize a composition by the fraction $t$ of its total intensity
assigned to negative jumps.  Its total intensity is
\[
 s(t)=\left(\frac{1-t}{a}+\frac t b\right)^{-1},\qquad 0\le t\le1,
\]
and $G_{a,b}=\max_t\Phi_{s(t),t}$.  If $t_*$ maximizes
$\Gamma_H$, then $1-t_*$ does too, by reflection of the two-sided profile.
Since
\[
 \frac1{s(t_*)}+\frac1{s(1-t_*)}=\frac1a+\frac1b=\frac2H,
\]
at least one of these two compositions, denoted $t$, has $s(t)\le H$.
At a fixed $t$, increasing the total intensity convolves both hypotheses
with common independent Skellam noise and contracts privacy.  Hence
$\Phi_{s(t),t}\ge\Phi_{H,t}=\Gamma_H$, as required.
\end{proof}

\begin{theorem}[Optimality among all binary-output channels]
\label{thm:binary-optimality}
If the infimum in \eqref{eq:operational-risk} is restricted to channels
having at most two output symbols, its value is $\lammin$.
\end{theorem}

\begin{proof}
Only the lower bound needs proof.  Choose a subsequence realizing a finite
risk liminf $r$. A one-output channel has $R_n=n^2/4$, so it cannot
occur along this subsequence for sufficiently large $n$.
By \eqref{eq:overlap-risk}, the overlap is $O(1/n)$.
For two probability rows on two symbols, overlap tending to zero forces
the rows to approach opposite point masses.  Relabel the two outputs
separately for each $n$ so that output one is the dominant output of input
one.  For all sufficiently large $n$, the overlap is then exactly
$W_0(1)+W_1(0)$.  Put
\[
 a_n=nW_0(1),\qquad b_n=nW_1(0).
\]
Both sequences are bounded, so extract limits $a_n\to a$, $b_n\to b$.
The two boundary counts and data processing give
$D_\eps(a),D_\eps(b)\le\delta$, so $a,b\ge\Lambda>0$.

For every fixed $0<\theta<1$, Lemma~\ref{lem:bayes-count} converges to
$(1-\theta)a+\theta b$.  Thus $r\ge\max(a,b)$, by taking the two
endpoint limits after $n\to\infty$.  Poisson approximation of the common
background, followed by coupling the changed user's dominant message,
is uniform over all neighboring compositions and gives the global limit
$G_{a,b}(\eps)$.  To be explicit, the two common-background binomials have
total TV error at most $(a_n^2+b_n^2)/n$. Coupling the changed message
and replacing the common-background factor $(n-1)/n$ by one cost at most
$2(a_n+b_n)/n$ together. Intensity convergence adds at most
$|a_n-a|+|b_n-b|$.  Coupling the limiting noise as the composition varies
also gives a Lipschitz bound, so passing from the composition grid to the
whole interval adds a vanishing error.  Hence $G_{a,b}(\eps)\le\delta$.
Lemma~\ref{lem:harmonic} now yields
\[
 r\ge\max(a,b)\ge\frac{2ab}{a+b}\ge\lammin.
\]
RR achievability proves equality.
\end{proof}

\section{A complete criterion for randomized-response optimality}
\label{sec:rr-optimality}

The preceding binary-output theorem does not imply optimality when observed
rare symbols are allowed.  We now characterize exactly when these extra
symbols can improve the risk.

\begin{theorem}[RR optimality criterion]
\label{thm:rr-classification}
For every $\eps\ge0$ and $0<\delta<1$,
\begin{equation}
 \Rstar=\lammin
 \quad\Longleftrightarrow\quad
 \Lambda=\lammin.
 \label{eq:rr-classification}
\end{equation}
When $\Lambda<\lammin$, a symmetric four-output channel is eventually
$(\eps,\delta)$-private and has limiting minimax risk strictly less than
$\lammin$.
\end{theorem}

\begin{proof}
If $\Lambda=\lammin$, Theorem~\ref{thm:boundary-converse} and RR
achievability give equality.  For the converse fix a strict gap and write
$\lambda_0=\lammin$.  Strict endpoint monotonicity gives
\[
 D_\eps(\lambda_0)<\delta=\Gamma_{\lambda_0}(\eps).
\]
We construct one perturbation that improves every globally worst
composition and every active orientation simultaneously.

\paragraph{Strict signed boundaries at every maximizer.}
Let $p_j=p_{\lambda,\theta}(j)$ be the Skellam mass and define the
forward signed sequence $g_j=p_{j-1}-hp_j$.  Log-concavity makes its
strictly positive set an upper ray.  At an interior global maximizing
composition, let $k$ be its first strictly positive index and set
$F_k=\sum_{j\ge k}g_j$.  This cutoff payoff is at most the global envelope
at every composition and equals it at the chosen one.  Consequently its
composition derivative vanishes there.  The Poisson generators give
\[
 \partial_\theta p_j=\lambda(p_{j+1}-p_{j-1}),\qquad
 \partial_\lambda p_j=(1-\theta)(p_{j-1}-p_j)
                         +\theta(p_{j+1}-p_j),
\]
and summation yields
\begin{equation}
 \partial_\theta F_k=-\lambda(g_{k-1}+g_k),\qquad
 \partial_\lambda F_k=(1-\theta)g_{k-1}-\theta g_k.
 \label{eq:rr-active-boundary}
\end{equation}
These differentiations follow termwise from the absolutely convergent
Poisson series; the displayed generators also bound the absolute sums of
the differentiated masses locally uniformly in $(\lambda,\theta)$.
Since $g_k>0$ and $g_{k-1}\le0$, stationarity implies
\begin{equation}
 g_k=-g_{k-1}=b>0,
 \qquad \partial_\lambda F_k=-b.
 \label{eq:rr-balanced-boundary}
\end{equation}
In particular a threshold tie $g_{k-1}=0$ is impossible at this maximizer.
The two strict boundary signs and log-concavity make this cutoff locally
unchanged, so the directed profile is smooth in a neighborhood.
For the reverse signed sequence $g_j=p_j-hp_{j-1}$, use its positive
lower ray $j\le k$.  The same argument gives
$g_k=-g_{k+1}=b>0$ and intensity derivative $-b$.
These statements hold separately for each active orientation.

\paragraph{The effect of one observed rare symbol.}
Add one rare output whose input rates are $(tu,tv)$, where $u,v>0$ and
$t$ is the perturbation parameter.  At composition $\theta$, its count
has distribution $\Poi(t\mu)$ with
\[
 \mu=(1-\theta)u+\theta v,
 \qquad q=\theta v/\mu.
\]
This count remains observed.  Conditional on exactly one such symbol, the
scalar noise is shifted by $-\operatorname{Ber}(q)$.  Its signed sequence
is therefore $(1-q)g_j+qg_{j+1}$.  In the forward direction, only the
boundary index $k-1$ can change sign.  Direct summation of positive masses
using \eqref{eq:rr-balanced-boundary} gives
\begin{equation}
 \delta_{\text{one symbol}}-\delta_{\text{no symbol}}
 =-b\min(q,1-q).
 \label{eq:rr-one-mark}
\end{equation}
The reverse direction has the same identity, with only index $k$ able to
change sign.  Multiplying by the first-order occurrence rate $\mu$ shows
that the profile derivative contributed by this observed symbol is
\begin{equation}
 -b\min\{(1-\theta)u,\theta v\}.
 \label{eq:rr-mark-derivative}
\end{equation}
There is no marginalization of the released rare-symbol count in this
calculation: the count has the same limiting law under both neighboring
hypotheses, and the directed profile is exactly the average of the
conditional directed profiles over this observed count law.

\paragraph{A common improvement for the full maximizer set.}
Let $\mathcal M$ be the set of composition-orientation pairs attaining
$\Gamma_{\lambda_0}(\eps)$.  It is compact, and its compositions have
positive distance from $0$ and $1$, because both endpoint profiles are
strictly smaller than $\delta$.  Choose $0<u<v$ so that every such
composition lies in
\[
 \left[\frac{u}{u+v},\frac{v}{u+v}\right].
\]
Add two distinct observed symbols with rate pairs $(tu,tv)$ and $(tv,tu)$.
At each maximizing composition,
\[
 \min\{(1-\theta)u,\theta v\}
 +\min\{(1-\theta)v,\theta u\}=u.
\]
Put $M=uv/(u+v)<u$, choose $M<\kappa<u$, and reduce the cross-error
intensity to $a(t)=\lambda_0-\kappa t$.  At every pair in $\mathcal M$,
the derivative of its perturbed directed profile is
\begin{equation}
 (\kappa-u)b<0.
 \label{eq:rr-common-improvement}
\end{equation}

We justify the passage from these derivatives to the global envelope.
The total rare-symbol count is $\Poi(t(u+v))$, uniformly in composition,
and its probability of at least two events is uniformly $O(t^2)$.
For each active composition-orientation pair, the strict boundary signs
give a neighborhood in $(\theta,\lambda)$ on which the zero-event
directed profile has a fixed cutoff and is smooth. Choose finitely many
such neighborhoods covering $\mathcal M$, and shrink their composition
neighborhoods so that their closures remain within the smooth regions
for all sufficiently small $t$, with $\lambda=\lambda_0-\kappa t$.
The one-event conditional profiles are jointly continuous and Lipschitz
in intensity, uniformly over these compact neighborhoods, by TV coupling.
Conditioning on zero, one, or at least two marks therefore gives
\[
 G_o(\theta,t)=H_o(\theta,\lambda_0)
                 +tC_o(\theta)+O(t^2)
\]
uniformly there. Here $o$ is a direction, $G_o$ its perturbed profile,
$H_o$ its unperturbed profile, and $C_o$ is continuous. On $\mathcal M$,
$C_o=(\kappa-u)b<0$. The locally fixed cutoff makes $b$ continuous on
$\mathcal M$; its strict positivity proved above and compactness give
$\inf_{\mathcal M}b>0$. Shrinking the
neighborhoods again makes $C_o$ uniformly negative on their closures.
On the slice $\lambda=\lambda_0$, we have
$H_o(\theta,\lambda_0)\le\delta$ for every composition and direction.
Consequently the perturbed profiles on the chosen neighborhoods are
uniformly below $\delta$ for every sufficiently small positive $t$.
On the compact complement of these neighborhoods in
$[0,1]\times\{\text{forward},\text{reverse}\}$, the unperturbed profile
has a strict uniform gap below $\delta$, while the perturbation is
uniformly $O(t)$. Thus the \emph{global} limiting profile is strictly
below $\delta$. This argument permits an arbitrary compact maximizer
set and either or both active directions. The bound by $\delta$ above is
used only on the slice $\lambda=\lambda_0$, not throughout a two-parameter
neighborhood.

\paragraph{Exact limiting risk and finite-$n$ privacy.}
For such a fixed $t>0$, take outputs $0,1,z_+,z_-$ and channel rows
\begin{align}
 W_0^{(n)}&=\left(1-\frac{a+t(u+v)}n,\frac an,
                          \frac{tu}n,\frac{tv}n\right),\nonumber\\
 W_1^{(n)}&=\left(\frac an,1-\frac{a+t(u+v)}n,
                          \frac{tv}n,\frac{tu}n\right),
 \qquad a=\lambda_0-\kappa t>0.
 \label{eq:rr-four-output}
\end{align}
Assign numerical values $0,1,v/(u+v),u/(u+v)$ to these outputs.  Under
input zero their mean and second moment are respectively $m/n$ and $r/n$,
where $m=a+2tM$ and $r=a+tM$; the input-one law is the reflection about
$1/2$.  If $S$ is the sum of these numerical values, then
\[
 \widehat k=\frac{S-m}{1-2m/n}
\]
is unbiased and has exactly the same MSE at every dataset:
\begin{equation}
 \frac{r-m^2/n}{(1-2m/n)^2}\longrightarrow r.
 \label{eq:rr-four-output-risk}
\end{equation}
The Bayes lower bound at $\theta=1/2$ is exactly
\[
 \frac{a[1-(a+t(u+v))/n]}{1-t(u+v)/n}+tM\longrightarrow r.
\]
Hence its actual minimax risk converges to
\[
 r=\lambda_0-(\kappa-M)t<\lambda_0.
\]
Theorem~\ref{thm:sparse-uniform} transfers the strict limiting
privacy margin to exact $(\eps,\delta)$-privacy for all sufficiently
large $n$.  The resulting sequence is admissible in
\eqref{eq:operational-risk} and proves $\Rstar<\lammin$.
\end{proof}

\begin{remark}[Scope of the construction]
The four-output channel is a marked mixture of two RR blocks.
The contribution here is the simultaneous global-profile perturbation
and its fixed-privacy minimax consequence.  Theorem~\ref{thm:rr-classification}
classifies RR optimality but does not provide a closed-form expression for
$\Rstar$ in the strict-gap regime.  Its strict improvement is existential;
the argument imposes no parameter-uniform lower bound on the admissible
perturbation size $t$.
\end{remark}

%% file: sections/strict_bounds.tex
\section{Universal lower bounds and strictness of the midpoint}
\label{sec:strict-bounds}

Throughout this section $\eps\ge0$ and $0<\delta<1$ are fixed. For
brevity write
\[
\ell=\Lambda(\eps,\delta),\qquad
\lambda_* =\lammin(\eps,\delta).
\]
These constants are positive and finite. All channel alphabets in the
universal results may grow arbitrarily with $n$.

\subsection{A boundary constraint on strong likelihood-ratio sets}

\begin{lemma}[Strong-set boundary intensities]
\label{lem:strong-boundary}
Suppose $\Delta_n(\eps)\le\delta$ and $R_n$ is bounded along a sequence
of sample sizes tending to infinity. For fixed $L>1$, let
\[
H_{n,L}=\{y:W_1(y)\ge LW_0(y)\},\qquad
Z_{n,L}=\{y:W_0(y)\ge LW_1(y)\},
\]
discarding symbols with both probabilities zero, and put
$a_{n,L}=nW_0(H_{n,L})$, $b_{n,L}=nW_1(Z_{n,L})$. Then
\[
\liminf_na_{n,L}\ge\ell,\qquad
\liminf_nb_{n,L}\ge\ell.
\]
\end{lemma}

\begin{proof}
By \eqref{eq:overlap-bound}, $nc_n$ is bounded. On $H_{n,L}$ the
minimum probability is $W_0$, so $a_{n,L}\le nc_n$; similarly
$b_{n,L}\le nc_n$. Outside $H_{n,L}$ we have
$W_1\le L\min(W_0,W_1)$, whence
\[
1-W_1(H_{n,L})\le Lc_n=O_L(n^{-1}).
\]
Under the all-zero dataset and its one-one neighbor, the count in
$H_{n,L}$ has laws
\[
\Bin(n,a_{n,L}/n),\qquad
\Bin(n-1,a_{n,L}/n)+\operatorname{Bern}(W_1(H_{n,L})),
\]
with independent summands in the second expression. Along any subsequence
where $a_{n,L}\to a$, the two laws converge in TV to
$\Poi(a)$ and $1+\Poi(a)$. Their privacy profiles are at most $\delta$
by postprocessing. Stability of the profile implies
$D_\eps(a)\le\delta$, so $a\ge\ell$. Boundedness of $a_{n,L}$ and
this property for every convergent subsequence imply its stated liminf
bound. Reversing the input labels and counting $Z_{n,L}$ proves the bound
for $b_{n,L}$.
\end{proof}

\subsection{The midpoint converse}

\begin{theorem}[Universal midpoint lower bound]
\label{thm:midpoint}
For every admissible sequence of common-channel, independent,
single-message binary shuffle mechanisms,
\[
\liminf_n R_n\ge
\frac{\Lambda(\eps,\delta)+\lammin(\eps,\delta)}2.
\]
Consequently
\[
\Rstar(\eps,\delta)\ge
\frac{\Lambda(\eps,\delta)+\lammin(\eps,\delta)}2.
\]
\end{theorem}

\begin{proof}
Take a subsequence realizing a finite $r=\liminf_nR_n$; the case
$r=\infty$ is immediate. All sufficiently large indices satisfy the
privacy requirement. Define the likelihood-ratio partition
$A_n=\{W_1\ge W_0\}$ and rates
\[
\alpha_n=nW_0(A_n),\qquad\beta_n=nW_1(A_n^c).
\]
Then $\alpha_n+\beta_n=nc_n\le4R_n$. Pass to a further subsequence on
which $\alpha_n\to\alpha$, $\beta_n\to\beta$. For each fixed $L>1$,
$H_{n,L}\subset A_n$ and $Z_{n,L}\subset A_n^c$ up to discarded zero
symbols. Lemma~\ref{lem:strong-boundary} therefore gives
$\alpha,\beta\ge\ell>0$.

At $\theta=1/2$ the exact Bayes formula gives
$B_n(1/2)=(n/2)\sum_y W_0W_1/(W_0+W_1)$. Globally the summand is at
least half the minimum probability. On either strong set it is at least
$L/(L+1)$ times the minimum. The two strong sets are disjoint after zero
symbols have been discarded. Consequently
\begin{equation}
R_n\ge\frac{\alpha_n+\beta_n}{4}
+\frac{L-1}{4(L+1)}(a_{n,L}+b_{n,L}).
\label{eq:bayes-bonus}
\end{equation}
Taking the limit inferior for fixed $L$ and then letting $L\to\infty$
proves
\begin{equation}
r\ge\frac{\alpha+\beta}{4}+\frac\ell2.
\label{eq:midpoint-first}
\end{equation}

The count in $A_n$ is a postprocessing of the histogram. At input count
$k$ its law is exactly
\[
k+\Bin(n-k,\alpha_n/n)-\Bin(k,\beta_n/n),
\]
with independent binomials. For each fixed $\theta\in[0,1]$, choose
neighboring counts whose composition tends to $\theta$. Recenter the
pair at its lower count and use the binomial-to-Poisson coupling. The
pair converges in TV to the asymmetric Skellam-shift pair with rates
$\alpha(1-\theta)$ and $\beta\theta$. Thus its profile is at most
$\delta$ for every $\theta$, including the boundary compositions, and
the global asymmetric envelope satisfies
$G_{\alpha,\beta}(\eps)\le\delta$.

By Lemma~\ref{lem:harmonic}, the harmonic mean
$2\alpha\beta/(\alpha+\beta)$ is at least $\lambda_*$. Hence the
arithmetic mean $(\alpha+\beta)/2$ is also at least $\lambda_*$.
Substituting into \eqref{eq:midpoint-first} proves the result. No uniform
approximation over growing alphabets was required: the preceding
composition limit was used separately for each fixed $\theta$.
\end{proof}

\subsection{Why equality at the midpoint is impossible in the gap regime}

\begin{theorem}[Strict midpoint converse]
\label{thm:strict-midpoint}
If $\Lambda(\eps,\delta)<\lammin(\eps,\delta)$, then
\[
\Rstar(\eps,\delta)>
\frac{\Lambda(\eps,\delta)+\lammin(\eps,\delta)}2.
\]
\end{theorem}

\begin{proof}
Write $M=(\ell+\lambda_*)/2$ and suppose first that an admissible array,
possibly considered only along a subsequence of sample sizes, has
$R_n\to M$. Put $u_y=nW_0(y)$, $v_y=nW_1(y)$, and associate with each
symbol its ratio $\rho_y=v_y/u_y$. Set $\rho_y=\infty$ if
$u_y=0<v_y$ and $\rho_y=0$ if $v_y=0<u_y$; simultaneous zero symbols
are immaterial. Symbols with zero minimum have zero weight in the overlap
measure but remain in every postprocessed channel below. The finite overlap measure
\[
m_n=\sum_y\min(u_y,v_y)\,\delta_{\rho_y}
\]
is defined on the compact extended half-line $[0,\infty]$. Its total mass
is $nc_n\le4R_n$. Passing to a subsequence gives weak convergence to a
finite measure $m$. The function
\[
\varphi(\rho)=\frac{\max(1,\rho)}{1+\rho}
\]
extends continuously to the compact half-line, with
$\varphi(0)=\varphi(\infty)=1$. The exact Bayes identity yields
\begin{equation}
B_n(1/2)=\frac12\int\varphi\,dm_n
\longrightarrow\frac12\int\varphi\,dm=:B_m(1/2).
\label{eq:overlap-measure-risk}
\end{equation}

For fixed $L>1$, the strong-set rate $a_{n,L}$ is
$m_n([L,\infty])$. By Lemma~\ref{lem:strong-boundary}, its limit
inferior is at least $\ell$. The set $[L,\infty]$ is closed, so the
Portmanteau inequality gives
\[
m([L,\infty])\ge\limsup_nm_n([L,\infty])\ge\ell.
\]
Letting $L\to\infty$ proves $m(\{\infty\})\ge\ell$.
The analogous closed sets $[0,1/L]$ give $m(\{0\})\ge\ell$.

For the MAP partition $\{\rho\ge1\}$, pass to a further subsequence
where its two error rates converge to $\alpha,\beta$. As in the proof of
Theorem~\ref{thm:midpoint}, its limiting envelope is at most $\delta$,
and Lemma~\ref{lem:harmonic} implies
$\alpha+\beta\ge2\lambda_*$. Since
$\alpha_n+\beta_n=m_n([0,\infty])$, weak convergence gives
$m([0,\infty])\ge2\lambda_*$.

On the interior half-line $\varphi\ge1/2$, with equality only at
$\rho=1$, while the endpoint values are one. Consequently
\begin{align*}
B_m(1/2)
&\ge\frac14m([0,\infty])
+\frac14\bigl(m(\{0\})+m(\{\infty\})\bigr)\\
&\ge\frac{\lambda_*+\ell}{2}=M.
\end{align*}
Together with $R_n\ge B_n(1/2)$ and \eqref{eq:overlap-measure-risk},
this forces equality throughout. Therefore the only possible limiting
measure is
\begin{equation}
m_* =\ell(\delta_0+\delta_\infty)
+2(\lambda_*-\ell)\delta_1.
\label{eq:midpoint-measure}
\end{equation}
The endpoint masses equal $\ell$, the total mass equals $2\lambda_*$,
and every interior mass must be concentrated at the unique minimum of
$\varphi$.

Now postprocess the original channel using the threshold
$A_n' =\{\rho_y\ge2\}$. Its error rates have the measure expressions
\[
\alpha_n'=\int_{[2,\infty]}dm_n,\qquad
\beta_n'=\int_{[0,2)}\max(1,\rho)\,dm_n.
\]
The integrands are bounded on their respective domains and are
continuous except at the threshold two. The measure in
\eqref{eq:midpoint-measure} has no atom there. Weak convergence thus
gives
\[
\alpha_n'\to\ell,\qquad\beta_n'\to2\lambda_*-\ell.
\]
Their binary neighboring pairs converge to the corresponding asymmetric
Skellam-shift pairs for each composition. Postprocessing and
Lemma~\ref{lem:harmonic} require
\[
\frac1\ell+\frac1{2\lambda_*-\ell}\le\frac2{\lambda_*}.
\]
This is impossible because
\[
\frac1\ell+\frac1{2\lambda_*-\ell}-\frac2{\lambda_*}
=\frac{2(\lambda_*-\ell)^2}
{\ell(2\lambda_*-\ell)\lambda_*}>0.
\]
We have excluded every admissible array or subsequence with risk
converging to $M$.

It remains to exclude equality of the infimum without attainment.
If $\Rstar(\eps,\delta)=M$, for each $j$ choose an admissible channel
sequence with limit inferior risk at most $M+1/j$. Choose an increasing
sample size $n_j$ sufficiently far along that sequence that its privacy
bound holds and its risk is at most $M+2/j$. The midpoint theorem also
forces the selected risks to have limit inferior at least $M$: one may
apply its proof directly to the selected finite arrays, or patch the
selected channels into a single sequence, using a fixed sufficiently
noisy critical RR channel at all other sample sizes. The latter channel
has eventual privacy with a strict limiting margin. The selected risks
therefore tend to $M$. Applying the contradiction just proved to this
selected array completes the argument.
\end{proof}

The strict inequality is qualitative. It gives neither a numerical
uniform gap above the midpoint nor a claim that any particular finite
marked channel attains the global infimum.

%% file: sections/variational.tex
\section{Reduction to finite marked channels}
\label{sec:variational}

The message alphabet of an admissible triangular array may grow without
restriction. Nevertheless, its asymptotic minimax value is an infimum over
fixed finite sparse channels. This reduction concerns the value of the
optimization problem; it does not assert equivalence of every original
histogram experiment to a fixed sparse experiment.

A finite marked configuration is
\[
 c=\bigl(a,b;(u_j,v_j)_{j=1}^J\bigr),
 \qquad a,b\ge0,\quad u_j,v_j>0,\quad J\ge0.
\]
Here $a,b$ are the two crossover intensities and $(u_j,v_j)$ are the
input-zero and input-one intensities of mark $j$. Let $B_c(\vartheta)$ be
the risk function in~\eqref{eq:finite-B}, and let
$\Psi_{c,\vartheta}(\eps)$ be the two-sided profile of the observed-mark
experiment in~\eqref{eq:marked-profile}. Thus
\begin{equation}
 \Vfin(\eps,\delta)
 =\inf_{c}\left\{
       \sup_{0\le\vartheta\le1}B_c(\vartheta):
       \sup_{0\le\vartheta\le1}\Psi_{c,\vartheta}(\eps)\le\delta
                 \right\}.
 \label{eq:finite-variational}
\end{equation}
The infimum ranges over all finite $J$ and all such rates. It is nonempty:
sufficiently noisy randomized response is a configuration with $J=0$.
Marks with a single positive rate, allowed in
Section~\ref{sec:finite-risk}, need not be included: merging such a mark
into the dominant output of its own input is a postprocessing that leaves
$B_c$ unchanged.

\begin{theorem}[Finite-channel variational reduction]
\label{thm:variational}
For every fixed $\eps\ge0$ and $0<\delta<1$, in the common-channel,
independent, single-message model of~\eqref{eq:operational-risk},
\[
 \boxed{\Rstar(\eps,\delta)=\Vfin(\eps,\delta).}
\]
In particular, the operational minimax can be approximated arbitrarily
closely by fixed finite marked configurations. No finite minimizer is
asserted.
\end{theorem}

\begin{proof}
We prove the two inequalities separately.

\paragraph{Finite configurations give admissible arrays.}
Fix a feasible configuration $c$ in~\eqref{eq:finite-variational} and
$\eta>0$. Add $\eta$ to both crossover intensities, obtaining $c^\eta$.
At composition $\vartheta$, this adds to the scalar noise an independent
difference of Poissons with means $(1-\vartheta)\eta$ and
$\vartheta\eta$, and leaves the observed mark counts unchanged. Hence it
is a common postprocessing of the two hypotheses. Moreover, the
contraction is strict whenever a directed profile is positive.

To verify strictness for $0<\vartheta<1$, condition on the observed mark
counts, and write $r$ for the conditional scalar mass. Either directed
profile is the positive mass of a summable signed sequence $g$ of the
form $r(j-1)-e^\eps r(j)$ or $r(j)-e^\eps r(j-1)$. Its total mass is
$1-e^\eps\le0$. If its positive mass is nonzero, $g$ has both positive
and negative entries. The added Skellam kernel $k$ is positive at every
integer. Thus
\[
 (g*k)_+(j)\le\sum_{\ell\in\mathbb Z}k(j-\ell)g_+(\ell),
\]
with a strict inequality at a coordinate receiving both signs. Summing
strictly decreases the positive mass. Averaging over the observed counts
preserves strictness for every positive unconditional directed profile.

At $\vartheta=0$, the marks have deterministic input-zero origin and are
independent of the scalar $\Poi(a)$ noise. At $\vartheta=1$, their origin
is again deterministic, and an observed-count recentering leaves the
reflected $\Poi(b)$ noise. Reflecting the scalar coordinate at this upper
endpoint also reflects its added increment, so it suffices in both cases
to consider addition of $\Poi(\eta)$ noise to an ordinary Poisson-shift
pair. These boundary profiles also strictly contract. Indeed, for a Poisson mass and
its unit shift, either directed signed sequence has entries of both
signs. A one-sided Poisson kernel assigns positive weight to both signs
at a sufficiently large output coordinate, so the same positive-part
argument applies. This argument also covers a zero initial mean.

A directed profile that is zero stays zero. Since $\delta>0$, it follows
that $\Psi_{c^\eta,\vartheta}(\eps)<\delta$ for every
$\vartheta\in[0,1]$. For a fixed finite configuration the profile is
continuous in $\vartheta$: coupling its finitely many latent Poisson
counts bounds the TV change by a constant times
$|\vartheta-\vartheta'|$, and profile stability then applies. Compactness
therefore gives
\[
 \sup_\vartheta\Psi_{c^\eta,\vartheta}(\eps)<\delta.
\]
Realize $c^\eta$ by a finite-output channel whose rare probabilities are
the specified intensities divided by $n$, and put the remaining row mass
at its own dominant output. Theorem~\ref{thm:sparse-uniform} gives
eventual $(\eps,\delta)$-DP. Theorem~\ref{thm:finite-risk} gives limiting
minimax risk
\[
 \sup_\vartheta B_{c^\eta}(\vartheta)
 =\sup_\vartheta B_c(\vartheta)+\eta.
\]
Letting $\eta\downarrow0$ and taking the infimum over $c$ proves
$\Rstar\le\Vfin$.

\paragraph{Compactness of overlap measures.}
For the reverse inequality, take an admissible sequence and set
$r=\liminf_n R_n$. If $r=\infty$ there is nothing to prove. Otherwise,
first choose a subsequence with $R_n\to r$. Put
\[
 u_{n,y}=nW_0^{(n)}(y),\qquad
 v_{n,y}=nW_1^{(n)}(y),\qquad
 \rho_{n,y}=v_{n,y}/u_{n,y}.
\]
When only $v_{n,y}$ is positive, set $\rho_{n,y}=\infty$; when only
$u_{n,y}$ is positive, set $\rho_{n,y}=0$. Symbols with both rates zero
can be omitted. Define the finite overlap measure
\[
 m_n=\sum_y\min(u_{n,y},v_{n,y})\,
                     \delta_{\rho_{n,y}}
\]
on the compact interval $[0,\infty]$, compactified through
$\rho\mapsto\rho/(1+\rho)$. By~\eqref{eq:overlap-bound},
$m_n([0,\infty])\le4R_n$. Along a further subsequence,
$m_n$ converges weakly to a finite measure $m$. Write
\[
 a=m(\{\infty\}),\qquad b=m(\{0\}),\qquad
 \mu=m|_{(0,\infty)}.
\]
The endpoint masses may be positive even though each $m_n$ gives zero
weight to pure one-input symbols.

For $0<\vartheta<1$, define the bounded continuous kernel
\[
 g_\vartheta(\rho)
 =\frac{\vartheta(1-\vartheta)\max(1,\rho)}
              {(1-\vartheta)+\vartheta\rho},\qquad
 g_\vartheta(0)=\vartheta,\quad
 g_\vartheta(\infty)=1-\vartheta,
\]
and put $B_m(\vartheta)=\int g_\vartheta\,dm$. The exact Bayes identity,
Lemma~\ref{lem:bayes-count}, reads
\[
 B_n(\vartheta)=\int g_\vartheta\,dm_n\longrightarrow
 B_m(\vartheta).
\]
Thus $r\ge B_m(\vartheta)$ for every fixed interior $\vartheta$.
Since $0\le g_\vartheta\le1$, dominated convergence extends $B_m$
continuously to the endpoints, with $B_m(0)=a$ and $B_m(1)=b$. Hence
\begin{equation}
 r\ge\sup_{0\le\vartheta\le1}B_m(\vartheta).
 \label{eq:variational-bayes-limit}
\end{equation}
Only pointwise Bayes convergence at fixed interior priors is used here;
the exact endpoint priors have zero finite-$n$ Bayes risk.

\paragraph{Finite coarsenings retain the privacy constraint.}
Fix $L>1$ such that $1/L$ and $L$ are continuity points of $m$, and
partition $[1/L,L)$ into finitely many intervals $I_1,\ldots,I_J$ whose
endpoints are also continuity points. Coarsen the original messages by
mapping ratios below $1/L$ to a dominant-zero symbol, ratios at least
$L$ to a dominant-one symbol, and each $I_j$ to a separate mark. Pure
input-zero and input-one outputs are included in their corresponding
dominant symbols, even though their overlap weights are zero.

The rare rates of this finite coarsened channel converge to
\begin{align*}
 a_L&=m([L,\infty]),&
 b_L&=m([0,1/L)),\\
 U_j&=\int_{I_j}\max(1,1/\rho)\,dm(\rho),&
 V_j&=\int_{I_j}\max(1,\rho)\,dm(\rho).
\end{align*}
The integrands restricted to a bin are bounded, with discontinuities
only at its endpoints, so weak convergence justifies these limits.
Zero-mass bins can be dropped; every remaining pair $(U_j,V_j)$ is
strictly positive.

To check the dominant probabilities, let
$c_n=\sum_y\min(W_0^{(n)}(y),W_1^{(n)}(y))$. Outside the high-ratio
set, $W_1^{(n)}\le L\min(W_0^{(n)},W_1^{(n)})$; consequently
\[
 1-W_1^{(n)}(\rho\ge L)\le Lc_n=O_L(n^{-1}).
\]
The analogous inequality holds for the low-ratio dominant-zero set.
All non-dominant rates have finite limits, and the dominant row totals
are determined by subtracting those rare probabilities from one. Thus
the limiting configuration
$c_{L,\mathcal I}=(a_L,b_L;(U_j,V_j)_j)$ satisfies the hypotheses of
Theorem~\ref{thm:sparse-uniform}.

Coarsening is a postprocessing, so the finite-$n$ coarsened profiles are
at most $\delta$. For each fixed composition, including either
endpoint, choose neighboring dataset compositions converging to it and
apply the sparse TV approximation. It follows that
\[
 \Psi_{c_{L,\mathcal I},\vartheta}(\eps)\le\delta
 \quad\text{for every }\vartheta\in[0,1].
\]
Hence $c_{L,\mathcal I}$ is a feasible configuration in
\eqref{eq:finite-variational}. No full limiting experiment for the
original growing alphabet is needed.

\paragraph{Uniform approximation of the risk functional.}
The endpoint atoms of $m$ are retained exactly. Replacing the finite
high-ratio tail by crossover noise increases its contribution by at
most $\mu([L,\infty))$, uniformly in $\vartheta$; the corresponding
low-tail increase is at most $\mu((0,1/L))$. Indeed,
$0\le g_\vartheta(\rho)\le1-\vartheta$ for $\rho\ge1$, and
$0\le g_\vartheta(\rho)\le\vartheta$ for $\rho\le1$. Both tail masses tend
to zero as $L\to\infty$ through the selected continuity cutoffs.

For the interior bins, put
\[
 u(\rho)=\max(1,1/\rho),\qquad v(\rho)=\max(1,\rho),\qquad
 F_\vartheta(x,y)=
 \frac{\vartheta(1-\vartheta)xy}{(1-\vartheta)x+\vartheta y}.
\]
The function $F$ is jointly uniformly continuous on
$[0,1]\times[1,L]^2$. In a positive-mass bin, let
$t_j=\mu(I_j)$ and let $\bar u_j,\bar v_j$ be the $\mu$-averages of
$u,v$ on that bin. Homogeneity gives its coarsened risk contribution as
\[
 F_\vartheta(U_j,V_j)=t_jF_\vartheta(\bar u_j,\bar v_j),
\]
whereas its original contribution is
$\int_{I_j}F_\vartheta(u(\rho),v(\rho))\,d\mu(\rho)$.
For a uniform continuity modulus $\omega_L$ of $F$, their difference in
absolute value is therefore at most
\[
 t_j\omega_L\bigl(\operatorname{osc}_{I_j}u+
                         \operatorname{osc}_{I_j}v\bigr),
\]
uniformly over $\vartheta$. Since $u,v$ are continuous throughout
$[1/L,L]$, including at one, the summed bound tends to zero when the
partition mesh tends to zero. A bin may straddle one; there is no need
to make one a partition boundary when $m$ has an atom there.

It follows that for every $\zeta>0$ one can choose $L$ and a finite
partition with continuity endpoints such that
\[
 \sup_\vartheta B_{c_{L,\mathcal I}}(\vartheta)
 \le\sup_\vartheta B_m(\vartheta)+\zeta
 \le r+\zeta.
\]
This feasible finite configuration shows $\Vfin\le r+\zeta$.
Letting $\zeta\downarrow0$ and taking the infimum over admissible
sequences proves $\Vfin\le\Rstar$, and completes the proof.
\end{proof}

Theorem~\ref{thm:variational} gives an exact variational description,
with observed finite mark counts retained in the privacy constraint.
It does not evaluate the infimum in closed form or establish that an
optimizer has finite support. In particular, finite-channel numerical
optimization supplies candidates for upper bounds, rather than a
certificate of global optimality.

%% file: sections/certified_example.tex
\section{A finite-sample privacy certificate and asymptotic risk separation}
\label{sec:certified-example}

The following example supplies an explicit separation with a finite-$n$
privacy guarantee. The channel is augmented randomized response with a
common null output. Its risk separation from binary randomized response
is asymptotic.

\begin{proposition}[A certified ternary example]
\label{prop:certified-example}
Set $\varepsilon=1/1000$, $\delta=23/50$, $a=9/10$, and $c=2/5$.
For $n\ge2500$, the channel on $\{0,1,\bot\}$ with rows
\[
 W_0^{(n)}=\left(1-\frac{a+c}{n},\frac an,\frac cn\right),
 \qquad
 W_1^{(n)}=\left(\frac an,1-\frac{a+c}{n},\frac cn\right)
\]
is globally $(\varepsilon,\delta)$-shuffle private.  Its minimax count
risk converges to $1$, whereas
\[
 \lambda_{\min}(1/1000,23/50)>\frac{203}{200}=1.015.
\]
Consequently $R_*(1/1000,23/50)\le1<\lambda_{\min}(1/1000,23/50)$.
\end{proposition}

\begin{proof}
Write $N_1$ and $C$ for the output-$1$ and null counts.  Coding a
message as $Z=0,1,1/2$, respectively, gives, under input zero,
\[
 \E Z=\frac{a+c/2}{n},\qquad
 \E Z^2=\frac{a+c/4}{n}.
\]
The input-one law is the reflection $1-Z$.  Thus the estimate
\[
 \widehat k=
 \frac{N_1+C/2-11/10}{1-11/(5n)}
\]
is unbiased at every $k$, with constant squared-error risk
\begin{equation}
 \E_k(\widehat k-k)^2
 =\frac{1-121/(100n)}{(1-11/(5n))^2}\longrightarrow1.
 \label{eq:certificate-risk}
\end{equation}
The Bernoulli-$1/2$ Bayes lower bound is exactly
\[
 a\frac{1-(a+c)/n}{1-c/n}+\frac c4\longrightarrow1.
\]
This proves convergence of the channel's minimax risk, with arbitrary
estimators allowed.

For composition $\theta$, its limiting neighboring pair consists of
$(S_\theta,C)$ and $(S_\theta+1,C)$, where
\[
 A\sim\Poi(a(1-\theta)),\quad B\sim\Poi(a\theta),\quad
 C\sim\Poi(c),\qquad
 S_\theta=A-B-V,\quad V\mid C\sim\Bin(C,\theta),
\]
with $A,B$ independent of $(C,V)$.  Let
\[
 H(0)=\sup_{0\le\theta\le1}
 \TV\bigl(\Law(S_\theta,C),\Law(S_\theta+1,C)\bigr).
\]
The null count is observed.  Since its distribution is identical under
the two hypotheses, the TV at a fixed composition equals
\begin{equation}
 \sum_{m\ge0}\Pr(C=m)\,
 \TV\bigl(\Law(A-B-\Bin(m,\theta)),
            \Law(1+A-B-\Bin(m,\theta))\bigr).
 \label{eq:certificate-conditional-tv}
\end{equation}
In particular, this calculation retains all observed rare marks.

Here is the numerical certificate and its analytic error control.
A marked-Poisson coupling changes the law of $(S_\theta,C)$ by at most
$(a+c)|\theta-\theta'|$ in TV when the composition changes.  Applying
the same bound to its translate shows that the TV in
\eqref{eq:certificate-conditional-tv} is $2(a+c)=13/5$-Lipschitz.
Reflection makes it symmetric about $1/2$.  Therefore the grid
$\theta=j/1000$, $0\le j\le500$, requires a correction of only
\[
 \frac{13}{5}\frac{1}{2000}=0.0013.
\]
At each grid point, the supplied interval-arithmetic calculation retains
$A,B,C\le20$.  Every omitted event has $A+B+C>20$, and
$A+B+C\sim\Poi(13/10)$.  Chernoff's inequality and $e<3$ give
\begin{equation}
 \Pr(A+B+C>20)
 \le\left(\frac{39}{210}\right)^{21}
 <4.424\mathbin{\cdot}10^{-16}.
 \label{eq:certificate-tail}
\end{equation}
Deleting this mass from each hypothesis changes TV by at most twice
the displayed tail bound.  No event with all three counts at most $20$
is omitted.

The finite Poisson sums and Bernoulli convolutions use nonnegative
coefficients.  Each basic binary64 operation is rounded outward with
\texttt{nextafter}, using IEEE-754 correctly rounded arithmetic;
subtraction in the positive-part sums is also rounded upward.  The sole
transcendental input, $e^{-13/10}$, is enclosed by exact rational
Taylor bounds.  Explicitly, for
\[
 T_N(x)=\sum_{j=0}^N\frac{x^j}{j!},\qquad
 U_N(x)=T_N(x)+
 \frac{x^{N+1}}{(N+1)!}\frac{1}{1-x/(N+2)},
\]
one has $T_N(x)\le e^x\le U_N(x)$ when $0\le x<N+2$.
The certificate uses
$[1/U_{80}(13/10),1/T_{80}(13/10)]$, converted outward to binary64.
The largest certified truncated grid value is at most
$0.4549904127523490$, at $\theta=0.316$.
Adding the grid correction, twice the tail bound, and all arithmetic
rounding gives
\begin{equation}
 H(0)\le0.4562904127523500<0.46.
 \label{eq:certificate-global}
\end{equation}
The decimal upper bound in \eqref{eq:certificate-global} is rounded
upward from the returned binary64 upper endpoint.

For completeness, the finite-sample constant can be obtained directly
from the $n-1$ unaffected users.  Put $s=a+c=13/10$ and
$\theta=k/(n-1)$ for an adjacent pair with $k$ unaffected input ones.
Replacing their categorical rare-event counts by marked Poissons costs
at most $(n-1)s^2/n^2\le s^2/n$ in TV.  Replacing the distinguished
message by its dominant output costs at most $s/n$, and increasing the
background Poisson means from total intensity $(n-1)s/n$ to $s$ costs
at most $s/n$.  Each hypothesis is therefore within
\[
 \beta_n=\frac{s^2+2s}{n}=\frac{4.29}{n}
\]
of the corresponding limiting hypothesis, uniformly over all adjacent
compositions.  TV stability, \eqref{eq:certificate-global}, and
monotonicity of the privacy profile in $\varepsilon$ now yield
\[
 \Delta_n(1/1000)\le\Delta_n(0)
 \le0.4562904127523500+\frac{8.58}{n}<\frac{23}{50},
 \qquad n\ge2500.
\]
This is an upper bound over the full composition interval, so no
additional finite-composition grid correction is required here.

Finally, every directed privacy profile is $1$-Lipschitz in its
multiplier $h=e^\varepsilon$.  At $\lambda=203/200$, the exact
zero-parameter RR envelope gives
\[
 \Gamma_\lambda(1/1000)
 \ge e^{-\lambda}I_0(\lambda)-(e^{1/1000}-1)
 >0.4609257398593415>\frac{23}{50}.
\]
The last comparison uses only exact rational arithmetic: lower-bound
$I_0(\lambda)$ by its $31$ positive terms
$\sum_{j=0}^{30}(\lambda^2/4)^j/(j!)^2$, and upper-bound
$e^\lambda$ and $e^{1/1000}$ by $U_{80}$ above.  The resulting rational
lower bound exceeds $23/50$ by an integer cross-multiplication check.
Continuity and monotonicity of $\Gamma_\lambda(1/1000)$ imply the stated
strict lower bound on $\lambda_{\min}$.
\end{proof}

\begin{remark}[Reproducibility and scope of the certificate]
The programs \path{certify_ternary_rational.py} and
\path{root_certify_rr.py}, supplied as ancillary files of this arXiv
record (folder \path{anc/open_minimax}), together with their full JSON
outputs, reproduce the two numerical comparisons above.  The first certificate
assumes correctly rounded IEEE-754 binary64 basic operations; the
second is an exact-rational certificate.  The example proves an explicit
achievable risk and a separation from RR.  It does not identify the
minimax value or certify an optimizer of the variational problem.
\end{remark}

%% file: discussion.tex
\section{Discussion}
\label{sec:discussion}

The global privacy constraint distinguishes the critical counting problem
from boundary-only calibration. Observed rare marks can improve on the
binary RR benchmark precisely when its global calibration exceeds its
boundary calibration. Such improvement may be small, and the proof does
not give a perturbation size uniform in the privacy parameters.
Corollary~\ref{cor:small-target-gap} shows that the strict-gap regime
occurs at arbitrarily small fixed targets for every $\eps>0$.
It does not imply that RR is suboptimal at all sufficiently small
targets, or that its relative loss stays bounded away from zero.

The variational equality $\Rstar=\Vfin$ resolves the reduction to finite
designs at the level of infima. The remaining optimization problem is to
evaluate this value in the strict-gap region. Existence of an optimizer,
the necessary number of marks, and an explicit all-parameter formula
remain open. A numerical optimization over a finite grid of mark ratios
provides candidate designs, not a lower bound for the unrestricted
problem. A claimed lower certificate would require a valid outer
relaxation or a globally verified dual inequality.

The order of limits also matters. The operational statements fix
$(\eps,\delta)$ before taking $n\to\infty$. The small-$\delta$
expansion in Appendix~\ref{sec:tail} is sequential. It does not follow
that its substitution at an arbitrary target $\delta_n\downarrow0$
gives a valid finite-$n$ guarantee. The explicit additive approximation
errors quantify when such a relative-error conclusion is justified.

Finally, the reduction concerns a common independent single-message
channel. It leaves correlated randomization, user-dependent channels,
multiple-message communication, and higher-dimensional empirical
statistics for separate work. The growing output alphabet in the lower
bounds is not itself a novelty claim; the governing distinction is the
global privacy constraint and the unrestricted deterministic-count risk.

%% file: sections/composition_appendix.tex
\section{Composition optimization for the RR benchmark}
\label{sec:optimization}

\begin{theorem}[Exact global profile at zero]
\label{thm:zero-envelope}
For every $\lambda>0$,
\begin{equation}
\boxed{\Gamma_\lambda(0)=e^{-\lambda}I_0(\lambda),}
\label{eq:zero-envelope}
\end{equation}
and the unique maximizing composition is $\vartheta=1/2$. Here
$I_0(\lambda)=(2\pi)^{-1}\int_{-\pi}^{\pi}e^{\lambda\cos u}\,\dd u$.
\end{theorem}
\begin{proof}
By Lemma~\ref{lem:log-concavity}, $\Phi_{\lambda,\vartheta}(0)=\max_jp_{\lambda,\vartheta}(j)$. The characteristic function is
\[
\varphi_{\lambda,\vartheta}(u)
=\exp\{\lambda(\cos u-1)+i\lambda(1-2\vartheta)\sin u\}.
\]
Fourier inversion and the triangle inequality yield
\[
p_{\lambda,\vartheta}(j)\le\frac1{2\pi}\int_{-\pi}^{\pi}e^{\lambda(\cos u-1)}\,\dd u=e^{-\lambda}I_0(\lambda).
\]
At $\vartheta=1/2,j=0$, the phase is identically zero and equality holds. Equality elsewhere requires the continuous phase $\lambda(1-2\vartheta)\sin u-ju$ to be constant modulo $2\pi$ on the interval, hence constant. Its derivative then vanishes everywhere, forcing $\vartheta=1/2,j=0$. Since a summable mass function attains its largest atom, this also proves uniqueness of the maximizing composition.
\end{proof}

For example, at $\lambda=1$ the boundary profile is $e^{-1}=0.367879\ldots$, whereas the global value is $e^{-1}I_0(1)=0.465759\ldots$. Thus the homogeneous neighboring pair does not maximize the critical mechanism's distinguishability in general.

\begin{proposition}[An exact finite-population identity and its parity limitation]
\label{prop:odd-n}
If $n=2m+1$ and $0<\tau<1/2$, the global finite-$n$ profile at zero is
\[
\Delta_n(0)=(1-2\tau)\sum_{j=0}^m\binom mj^2\tau^{2j}(1-\tau)^{2m-2j},
\]
attained by $k=m$ common one-input users. Balanced optimality need not hold for even $n$.
\end{proposition}
\begin{proof}
From~\eqref{eq:exact-background}, $\TV(P,Q)=(1-2\tau)\TV(R,R^+)=(1-2\tau)\max_s R(s)$. The modulus of the characteristic function of the $2m$ common-user noise is $|1-\tau+\tau e^{iu}|^{2m}$, independent of $k$. Fourier inversion bounds every atom by the integral of this modulus. For $k=m$, the noise is the difference of two independent $\Bin(m,\tau)$ variables, so the integral is its zero atom, equal to the displayed sum.

For the limitation, take $n=8$, $\tau=3/10$. Exact common-background peaks are $631841/2000000$ at $k=3,4$ and $636069/2000000$ at $k=2,5$. Multiplication by $1-2\tau=2/5$ preserves the strict inequality. These fractions follow by expanding the seven Bernoulli factors; thus a balanced neighboring pair is not finite-$n$ worst in this example.
\end{proof}

\subsection{The two maximizing branches for small positive privacy parameters}
Write
\[
s_j=e^{-\lambda}I_j(\lambda)=p_{\lambda,1/2}(j),\qquad j\ge0.
\]
We will use $s_0>s_1>0$ and $s_0-s_2=2s_1/\lambda>0$. The first follows directly from
\[
s_0-s_1=\frac1{2\pi}\int_{-\pi}^{\pi}e^{\lambda(\cos u-1)}(1-\cos u)\,\dd u>0,
\]
and the second follows from the coefficient recurrence proved in Section~\ref{sec:tail}, at $j=1$ and $\vartheta=1/2$.

\begin{theorem}[Small-positive-parameter composition optimization]
\label{thm:small-eps}
Fix $\lambda>0$ and put $t=e^\eps-1$. For all sufficiently small $t>0$, the global envelope has exactly two maximizing compositions, and
\begin{align}
\vartheta_\pm(t)&=\frac12\pm A_\lambda t+O_\lambda(t^2),
&A_\lambda&=\frac{s_0+s_1}{4s_1},\label{eq:optimizer-expansion}\\
\Gamma_\lambda(\eps)&=s_0-\frac{1-s_0}{2}t+B_\lambda t^2+O_\lambda(t^3),
&B_\lambda&=\frac{\lambda(s_0+s_1)^2}{8s_1}>0.\label{eq:envelope-expansion}
\end{align}
The maximizers are symmetric, strictly interior, and different from $1/2$.
\end{theorem}
\begin{proof}
Let $F_\vartheta$ be the CDF of $D_{\lambda,\vartheta}$. At $(\vartheta,t)=(1/2,0)$ the reverse-profile active set is exactly $j\le0$: its ratio $p(j)/p(j-1)$ exceeds one there and is less than one for $j\ge1$. More precisely, the adjacent ratios are $s_0/s_1>1$ at $j=0$ and $s_1/s_0<1$ at $j=1$. By parameter continuity these inequalities persist against $1+t$ in a neighborhood of $(1/2,0)$. Log-concavity makes the ratios nonincreasing in $j$, so throughout that neighborhood the reverse profile is the analytic function
\begin{equation}
R(\vartheta,t)=F_\vartheta(0)-(1+t)F_\vartheta(-1).
\label{eq:reverse-branch}
\end{equation}
The forward profile is $R(1-\vartheta,t)$ by~\eqref{eq:reflection}. Analyticity follows, for example, by conditioning on a $\Poi(\lambda)$ total jump count: the finite conditional probabilities are polynomials in $\vartheta$, and their Poisson-weighted series and all derivatives converge uniformly in a complex neighborhood of $1/2$.

Differentiating the characteristic function, or the Poisson series, gives
\[
\partial_\vartheta p_{\lambda,\vartheta}(j)
=\lambda[p_{\lambda,\vartheta}(j+1)-p_{\lambda,\vartheta}(j-1)],\qquad
\partial_\vartheta F_\vartheta(j)
=\lambda[p_{\lambda,\vartheta}(j)+p_{\lambda,\vartheta}(j+1)].
\]
Consequently,
\begin{equation}
R_\vartheta(1/2,t)=-\lambda t(s_0+s_1),\qquad
R_{\vartheta\vartheta}(1/2,0)=-2\lambda^2(s_0-s_2)<0.
\label{eq:branch-derivatives}
\end{equation}
Choose a closed interval $U$ around $1/2$ contained in the threshold neighborhood and small enough that $R_{\vartheta\vartheta}<0$ throughout $U$ for small $|t|$. The implicit function theorem yields a unique stationary point of $R$ there, which is its strict maximum, with
\[
\vartheta_-(t)=\frac12-\frac{s_0+s_1}{2\lambda(s_0-s_2)}t+O_\lambda(t^2).
\]
Its reflected point $\vartheta_+(t)=1-\vartheta_-(t)$ uniquely maximizes the forward branch.

To justify globality and exhaustiveness, Theorem~\ref{thm:zero-envelope} gives a strict gap between its unique maximum and the maximum on the compact complement of the interior of $U$. For every probability pair the profile changes by at most $t$ when the multiplier is changed from $1$ to $1+t$. This uniform bound preserves the gap for small $t>0$. Hence every global maximizer belongs to $U$. Within $U$, the maximum of the two branches attains its largest value exactly at their two distinct branch maxima: reflection makes these largest values equal, and strict concavity makes each branch maximum unique. Thus there are exactly two global maximizers.

Finally $R(1/2,t)=s_0-t(1-s_0)/2$. Taylor expansion at its stationary point using~\eqref{eq:branch-derivatives} gives the gain $(s_0+s_1)^2/[4(s_0-s_2)]\,t^2+O_\lambda(t^3)$. Substituting $s_0-s_2=2s_1/\lambda$ proves both displayed formulas.
\end{proof}

\begin{corollary}[RR suboptimality near zero privacy parameter]
\label{cor:near-zero-gap}
For every $0<\delta<1$, there is $\eta_\delta>0$ such that
\[
 \Lambda(\eps,\delta)<\lammin(\eps,\delta)
 \qquad\text{for }0\le\eps<\eta_\delta.
\]
Consequently, binary RR is strictly suboptimal throughout this interval.
For every compact interval $K\subset(0,1)$, the value of $\eta$ can be
chosen uniformly for $\delta\in K$.
\end{corollary}

\begin{proof}
The unique maximizing composition in Theorem~\ref{thm:zero-envelope}
is $1/2$, so $D_0(\lambda)<\Gamma_\lambda(0)$ for every $\lambda>0$.
At the positive root $\lambda=\Lambda(0,\delta)$, this gives
$\Gamma_\lambda(0)>\delta$. Theorem~\ref{thm:calibration} therefore
implies $\Lambda(0,\delta)<\lammin(0,\delta)$.

For either profile $F(\eps,\lambda)=D_\eps(\lambda)$ or
$F(\eps,\lambda)=\Gamma_\lambda(\eps)$, Poisson coupling and the
multiplier bound give
\[
 |F(\eps,\lambda)-F(\eps',\lambda')|
 \le (1+e^{\max(\eps,\eps')})|\lambda-\lambda'|
        +|e^\eps-e^{\eps'}|.
\]
Thus both profiles are jointly continuous. Their unique level roots
are jointly continuous in $(\eps,\delta)$: at a root $r_0>0$, choose
$0\le a<r_0<b$ as close to $r_0$ as desired. Strict monotonicity gives
$F(\eps_0,a)>\delta_0>F(\eps_0,b)$; these inequalities persist for
nearby parameters and trap the corresponding roots between $a$ and $b$.
The continuous gap $\lammin-\Lambda$ is strictly positive on
$\{0\}\times K$. Compactness preserves its positivity on
$[0,\eta)\times K$ for some $\eta>0$. A singleton $K=\{\delta\}$
gives the first assertion, and Theorem~\ref{thm:rr-classification}
gives strict RR suboptimality.
\end{proof}

\subsection{A certified intermediate-composition example}
At $\lambda=1$, $e^\eps=3/2$, the values of the limiting profiles are
\begin{center}
\begin{tabular}{rrrr}
\toprule
$\vartheta$ & Forward & Reverse & Two-sided\\
\midrule
$0$ & $0.235758882343$ & $0.367879441171$ & $0.367879441171$\\
$1/5$ & $0.236252495154$ & $0.385178976834$ & $0.385178976834$\\
$1/2$ & $0.332199509492$ & $0.332199509492$ & $0.332199509492$\\
\bottomrule
\end{tabular}
\end{center}
In particular, $\Gamma_1(\log(3/2))$ is larger than both the boundary and balanced values. This comparison can be certified without floating-point Skellam evaluation. Truncate the joint Poisson sum to $u+v\le M$. Before the common factor $e^{-1}$, each mass is
\[
\frac{(1-\vartheta)^u\vartheta^v}{u!v!},
\]
which is rational for these three compositions. The omitted unnormalized mass is at most
\[
R_M=\frac{1}{(M+1)!}\frac{1}{1-1/(M+2)}.
\]
Each directed truncated profile differs from its true value by at most $\frac52e^{-1}R_M$. At $M=12$, exact rational summation gives $e^{-1}$ times a value in $(1.047,1.048)$ for the intermediate reverse profile, while the boundary reverse value is $e^{-1}$, its forward raw profile is in $(0.640,0.642)$, and both balanced raw profiles are in $(0.903,0.904)$; $5R_{12}<10^{-8}$. These rational intervals, together with the tail bound, prove the strict comparisons. The displayed decimals are rounded numerical evaluations.

\begin{figure}[t]
\centering
\includegraphics[width=.92\linewidth]{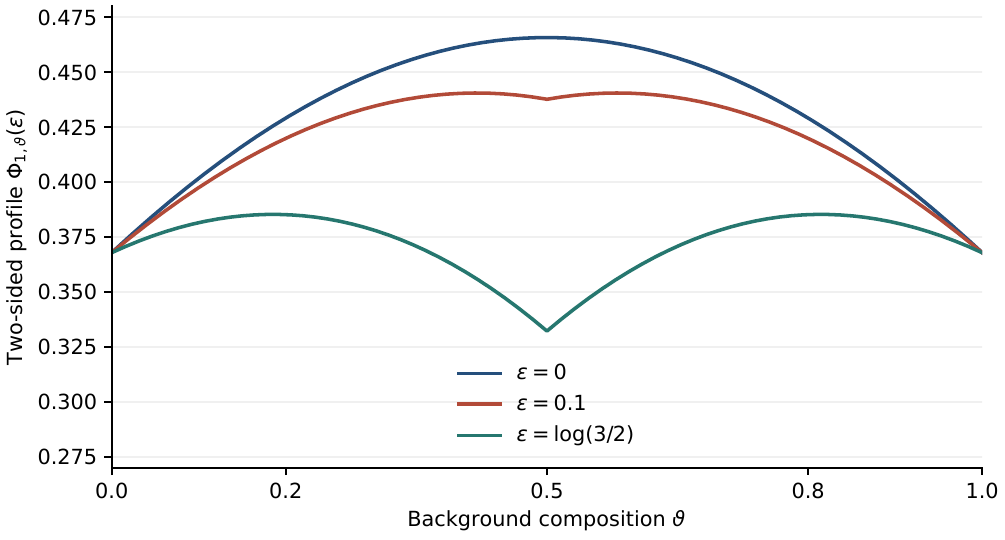}
\caption{Two-sided Skellam-shift privacy profiles at $\lambda=1$ as functions of composition. The zero-parameter maximum is balanced; a positive parameter moves the worst compositions away from the center. Curves are numerical evaluations of the explicit series. The small-parameter branch result is proved in Theorem~\ref{thm:small-eps}; the intermediate-composition comparison at $\eps=\log(3/2)$ has the rational certificate described in the text.}
\label{fig:composition}
\end{figure}
\FloatBarrier

%% file: sections/tail_appendix.tex
\section{Global non-Gaussian tails and calibration}
\label{sec:tail}

Put $I(x)=x\log x-x+1$, $x>0$. Exponential privacy-tail bounds for the symmetric additive Skellam mechanism were developed in~\cite{VA18}, and a tight numerical accountant for the scalar symmetric Skellam mechanism appears in~\cite{AKL21}. Our question is the maximum over the asymmetric composition parameter. For fixed $\eps>0$ the Gaussian rate $\eps^2/2$ is not the correct worst-composition rate. We first obtain a uniform bound by controlling the likelihood ratio exactly.

\begin{lemma}[An exact Skellam likelihood-ratio bracket]
\label{lem:ratio-bracket}
Let $a,b>0$, $a+b=1$, $D=X-Y$ with Poisson means $\lambda a,\lambda b$, and $p_j=\Prob(D=j)$. With $r_j=p_{j-1}/p_j$,
\begin{equation}
\boxed{j-1\le\lambda(ar_j-b/r_j)\le j,\qquad j\in\mathbb Z.}
\label{eq:ratio-bracket}
\end{equation}
\end{lemma}
\begin{proof}
Poisson integration by parts gives
\[
jp_j=\lambda a p_{j-1}-\lambda b p_{j+1}.
\]
Indeed, $\E[X\mathbf1\{X-Y=j\}]=\lambda a p_{j-1}$ and the corresponding $Y$ expression is $\lambda b p_{j+1}$. Log-concavity implies $r_{j-1}\le r_j\le r_{j+1}$. The recurrence at $j$ yields
$j=\lambda(ar_j-b/r_{j+1})\ge\lambda(ar_j-b/r_j)$; at $j-1$ it yields
$j-1=\lambda(ar_{j-1}-b/r_j)\le\lambda(ar_j-b/r_j)$. These are exactly~\eqref{eq:ratio-bracket}.
\end{proof}

\begin{lemma}[A uniform tilted atom bound]
\label{lem:atom}
Any Skellam law with total Poisson intensity $u>0$ has
\[
\max_jp_j\le\frac{\sqrt\pi}{2\sqrt{2u}}.
\]
\end{lemma}
\begin{proof}
Fourier inversion bounds any atom by $(2\pi)^{-1}\int_{-\pi}^{\pi}e^{-u(1-\cos z)}\,\dd z$. On this interval $1-\cos z\ge2z^2/\pi^2$. Extending the Gaussian integral to the real line gives the stated constant.
\end{proof}

\begin{theorem}[Global large-noise privacy order]
\label{thm:global-tail}
For every fixed $\eps>0$, as $\lambda\to\infty$,
\begin{equation}
\boxed{\Gamma_\lambda(\eps)=\Theta_\eps\left(\lambda^{-3/2}e^{-\lambda I(e^{-\eps})}\right).}
\label{eq:global-tail}
\end{equation}
For all $\lambda>0$, a valid upper bound is
\begin{equation}
\Gamma_\lambda(\eps)\le C_\eps\lambda^{-3/2}e^{-\lambda I(e^{-\eps})},\qquad
C_\eps=\sqrt{\frac\pi2}\frac{e^{5\eps/2}}{(1-e^{-\eps})^2}.
\label{eq:tail-upper}
\end{equation}
For all sufficiently large $\lambda$, a valid lower constant in~\eqref{eq:global-tail} is $c_\eps=e^{-\eps/2}/(2\sqrt{2\pi})$.
\end{theorem}
\begin{proof}
First take an interior composition, write $a=1-\vartheta$, $b=\vartheta$, $h=e^\eps$, and
\[
\psi(\theta)=a(e^\theta-1)+b(e^{-\theta}-1),\qquad
T=\lambda\psi'(\eps)=\lambda(ah-b/h).
\]
A forward positive summand has $r_j>h$. Since $f(r)=ar-b/r$ is increasing, Lemma~\ref{lem:ratio-bracket} implies $s=j-T>0$ and
\begin{equation}
(r_j-h)\left(a+\frac{b}{r_jh}\right)\le\frac{s}{\lambda}.
\label{eq:payoff-bound}
\end{equation}
If $a\ge1/2$, then $(p_{j-1}-hp_j)_+\le2sp_j/\lambda$. If $b\ge1/2$, dividing~\eqref{eq:payoff-bound} by $r_j$ gives $(p_{j-1}-hp_j)_+\le2hs p_{j-1}/\lambda$. These cases cover every interior composition.

Tilting the mass function by $e^{\eps j}$ gives a Skellam law with means $\lambda ah,\lambda b/h$. Denote its mass by $\widetilde p_j$ and put
\[
v=ah+b/h\ge h^{-1},\qquad
J_\vartheta(\eps)=\eps\psi'(\eps)-\psi(\eps)
=aI(h)+bI(h^{-1}).
\]
The exact tilt identity is
\[
p_j=e^{-\lambda J_\vartheta(\eps)}e^{-\eps(j-T)}\widetilde p_j.
\]
For $p_{j-1}$ the same identity has an extra factor $h$ and the tilted atom at $j-1$. Lemma~\ref{lem:atom} and the two payoff cases therefore yield
\[
\delta_{Q_{\lambda,\vartheta}\Vert P_{\lambda,\vartheta}}(\eps)
\le\frac{2h^2}{\lambda}e^{-\lambda J_\vartheta(\eps)}
\frac{\sqrt\pi}{2\sqrt{2\lambda v}}
\sum_{j>T}(j-T)e^{-\eps(j-T)}.
\]
The positive values $j-T$ are $\eta+m$ with $0<\eta\le1$, $m\ge0$, so the sum is at most $\sum_{m\ge0}(m+1)e^{-\eps m}=(1-e^{-\eps})^{-2}$. We obtain~\eqref{eq:tail-upper} with $J_\vartheta(\eps)$ in place of $I(e^{-\eps})$ for the directed profile. Reflection gives the reverse bound with $\vartheta$ replaced by $1-\vartheta$.

For $\eps>0$,
\[
I(e^\eps)-I(e^{-\eps})=2(\eps\cosh\eps-\sinh\eps)>0,
\]
because its derivative is $2\eps\sinh\eps>0$ and it vanishes at zero. Hence both directed rates are at least $I(e^{-\eps})$. Continuity in composition from Lemma~\ref{lem:parameter} extends the uniform upper bound to the endpoints and proves~\eqref{eq:tail-upper} after maximization.

For the lower bound, take $\vartheta=0$, $x=e^{-\eps}$ and $j=\lfloor\lambda x\rfloor-1\ge0$. Write $\rho=\lambda x-j\in[1,2)$. The reverse summand at $j$ is
\[
p_j\left(1-\frac{j}{\lambda x}\right)=p_j\frac{\rho}{\lambda x}.
\]
Uniformly over $\rho\in[1,2)$, Stirling's formula gives
\[
p_j=\frac{x^\rho}{\sqrt{2\pi\lambda x}}e^{-\lambda I(x)}(1+O_\eps(\lambda^{-1})).
\]
For completeness, this follows by expanding $I(j/\lambda)$ about $x$: $\lambda I(j/\lambda)=\lambda I(x)-\rho\log x+O_\eps(\lambda^{-1})$, and the Poisson atom has prefactor $(2\pi j)^{-1/2}(1+O(j^{-1}))$. For sufficiently large $\lambda$, the atom is at least half its leading term. Since $\rho\ge1$ and $x^\rho\ge x^2$, the positive summand is at least $\sqrt{x}/(2\sqrt{2\pi})$ times $\lambda^{-3/2}e^{-\lambda I(x)}$. This proves the lower bound and the theorem.
\end{proof}

\begin{corollary}[Boundary approach at the worst exponential rate]
\label{cor:boundary-tail}
For every fixed $\eta\in(0,1/2)$ and $\eps>0$, compositions $\vartheta\in[\eta,1-\eta]$ satisfy
\[
\Phi_{\lambda,\vartheta}(\eps)\le C_\eps\lambda^{-3/2}
\exp\{-\lambda[I(e^{-\eps})+\eta(I(e^\eps)-I(e^{-\eps}))]\}.
\]
Thus any sequence of maximizing compositions for $\lambda\to\infty$ must approach the boundary set $\{0,1\}$. This does not assert exact boundary maximality at every finite $\lambda$.
\end{corollary}
\begin{proof}
The two directed rates in the proof of Theorem~\ref{thm:global-tail} are convex combinations of $I(e^\eps)$ and $I(e^{-\eps})$, with the weight on the former at least $\eta$. Compare with the boundary lower bound.
\end{proof}

\subsection{Calibration and randomized-response utility}

\begin{corollary}[Sequential small-target calibration and minimax risk]
\label{cor:inversion}
For fixed $\eps>0$, each of $\Lambda(\eps,\delta)$,
$\lammin(\eps,\delta)$, and $\Rstar(\eps,\delta)$ equals
\begin{equation}
 \frac{\log(1/\delta)-\tfrac32\log\log(1/\delta)}
 {1-(1+\eps)e^{-\eps}}+O_\eps(1),\qquad \delta\downarrow0.
 \label{eq:inverse}
\end{equation}
\end{corollary}
\begin{proof}
The lower bound in Theorem~\ref{thm:global-tail} is a single summand of
the boundary Poisson profile. Its upper bound holds for the whole envelope
and hence also for that boundary profile. Thus both $D_\eps(\lambda)$
and $\Gamma_\lambda(\eps)$ have order
$\lambda^{-3/2}e^{-\lambda I(e^{-\eps})}$.
By Theorem~\ref{thm:calibration}, each inverse is the unique level root.
Taking logarithms gives
$\log(1/\delta)=\lambda I(e^{-\eps})+\tfrac32\log\lambda+O_\eps(1)$.
First $\lambda=\Theta_\eps(\log(1/\delta))$; substitution then gives
\eqref{eq:inverse} for both roots. The bounds
$\Lambda\le\Rstar\le\lammin$ give the same expansion for $\Rstar$.
The denominator is positive for fixed $\eps>0$.
\end{proof}

This statement takes the critical population limit first. Under
$\tau_n=\lambda/(n+\lambda)$, the additive error in
Theorem~\ref{thm:global-envelope} is
$O_\eps((\lambda^2+\lambda)/n)$ when $\lambda$ varies.
A joint-limit relative approximation follows only when this error is
negligible compared with the target profile. In particular,
\eqref{eq:inverse} is not automatically a finite-$n$ guarantee at targets
of order $1/n$ or smaller.

%% file: sections/conditional_appendix.tex
\section{Conditional representation and support floors}
\label{sec:conditional-floors}
Let $\mathcal Z=\mathcal Y\setminus\{y_0,y_1\}$. For a fixed composition define
\[
a=(1-\vartheta)\alpha_0(y_1),\quad b=\vartheta\alpha_1(y_0),\qquad
u_y=(1-\vartheta)\alpha_0(y),\quad v_y=\vartheta\alpha_1(y),\quad y\in\mathcal Z.
\]
For zero total intensity $u_y+v_y=0$, set $Z_y=V_y=0$. Otherwise let $Z_y\sim\Poi(u_y+v_y)$ independently and, conditionally on $Z$, let $V_y\sim\Bin(Z_y,v_y/(u_y+v_y))$ independently. Independently take $A\sim\Poi(a)$, $B\sim\Poi(b)$ and put $D=A-B-\sum_{y\in\mathcal Z}V_y$.

\begin{proposition}[Conditional scalar experiment]
\label{prop:conditional}
The compound-Poisson pair $(\mathcal P_\vartheta,\mathcal Q_\vartheta)$ is exactly equivalent to observing $(D,Z)$ under the null and $(D+1,Z)$ under the alternative. If $r_z$ denotes the mass function of $D$ conditionally on $Z=z$, its directed profiles are
\begin{align*}
\delta_{\mathcal Q_\vartheta\Vert\mathcal P_\vartheta}(\eps)
&=\E_Z\sum_d(r_Z(d-1)-e^\eps r_Z(d))_+,\\
\delta_{\mathcal P_\vartheta\Vert\mathcal Q_\vartheta}(\eps)
&=\E_Z\sum_d(r_Z(d)-e^\eps r_Z(d-1))_+.
\end{align*}
The two-sided profile is the maximum of these two expectations.
\end{proposition}
\begin{proof}
In $H_\vartheta$, the released extra-symbol coordinate is $Z_y=U_y+V_y$. Poisson splitting gives the stated conditional binomial origin counts. The $y_1$ coordinate equals $A-B-\sum_yV_y=D$. The $y_0$ coordinate is recovered from the zero-sum constraint. The map $H\mapsto(D,Z)$ is therefore invertible on the supporting integer hyperplane. The neighbor shift increments $D$ and leaves $Z$ unchanged. Since $Z$ has a common distribution, the countable privacy identity applied to the joint masses proves the formulas.
\end{proof}

\begin{corollary}[Exact singular floors]
\label{cor:floors}
Let $S=\{y\in\mathcal Z:u_y>0,\ v_y>0\}$. Then
\begin{align}
\lim_{\eps\to\infty}\delta_{\mathcal P_\vartheta\Vert\mathcal Q_\vartheta}(\eps)
&=\begin{cases}0,&b>0,\\ \exp(-a-\sum_{y\in S}u_y),&b=0,\end{cases}\label{eq:reverse-floor}\\
\lim_{\eps\to\infty}\delta_{\mathcal Q_\vartheta\Vert\mathcal P_\vartheta}(\eps)
&=\begin{cases}0,&a>0,\\ \exp(-b-\sum_{y\in S}v_y),&a=0.\end{cases}\label{eq:forward-floor}
\end{align}
Thus both floors vanish exactly when $a,b>0$.
\end{corollary}
\begin{proof}
On a countable space, dominated convergence, with each summand bounded by the first measure's mass, leaves exactly the mass of points outside the other measure's support as $e^\eps\to\infty$. Conditional on $Z=z$, if $b>0$ the scalar null support extends without bound below; translating it by one introduces no null-only lower endpoint. If $a>0$ it extends without bound above and there is no alternative-only upper endpoint. If both are positive, the scalar support is all $\mathbb Z$.

When $b=0$, the unique null-only endpoint is the smallest possible $D$. Reaching it requires $A=0$ and $V_y=z_y$ for every shared-origin symbol; symbols with $u_y=0$ already have $V_y=z_y$, and symbols with $v_y=0$ have $V_y=0$. Its conditional probability is
\[
e^{-a}\prod_{y\in S}\left(\frac{v_y}{u_y+v_y}\right)^{z_y}.
\]
Averaging the Poisson generating functions gives $e^{-a-\sum_{y\in S}u_y}$. If also $a=0$, the support is a finite consecutive interval and the same endpoint argument still applies, including a singleton. The analogous upper-endpoint calculation when $a=0$ gives~\eqref{eq:forward-floor}.
\end{proof}

For instance, suppose both inputs emit an extra rare symbol at intensity one and the direct cross intensities vanish. At $\vartheta=1/2$, both singular floors are $e^{-1/2}$ even though the composition is interior. Positive finite-$n$ cross probabilities of order $n^{-2}$ realize this example with full support. The absence of a floor for interior \emph{binary RR} therefore does not extend without conditions to all sparse finite-output channels.

If no extra symbol can come from both input groups, the conditional $V_y$ are deterministic functions of $Z$. Adding $C(Z)=\sum_{y:u_y=0<v_y}Z_y$ to $D$ shows that the pair factors exactly into a scalar Skellam-shift experiment with means $a,b$ and common independent extra-symbol noise. In particular, the single-dominant boundary $\vartheta=0$ is a scalar Poisson-shift with mean $\alpha_0(y_1)$, and the upper boundary exchanges its orientation. This statement uses the single-dominant hypothesis; it is not a reduction theorem for a bounded population with multiple dominant outputs.

%% file: sections/plateau.tex
\section{The support-floor plateau and separation at small targets}
\label{app:plateau}

The endpoint singular atom gives $\Gamma_\lambda(\eps)\ge e^{-\lambda}$.
Here we identify exactly when the global envelope equals this floor.

\begin{theorem}[Exact plateau and its maximizing compositions]
\label{thm:floor-plateau}
For every $\lambda>0$ and $\eps\ge0$,
\begin{equation}
 \Gamma_\lambda(\eps)=e^{-\lambda}
 \quad\Longleftrightarrow\quad e^\eps\ge1+\lambda.
 \label{eq:floor-plateau}
\end{equation}
When $e^\eps\ge1+\lambda$, the maximizing compositions are exactly
$\{0,1\}$, including equality at the threshold.
\end{theorem}

\begin{proof}
Write $h=e^\eps$, $a=\lambda(1-\theta)$, and $b=\lambda\theta$.
First assume $h\ge1+\lambda$, $a\ge0$, and $b>0$.  Let $p_j$ be the
mass of $\Poi(a)-\Poi(b)$ and define, for integers $r\ge0$,
\[
 C_r(a,b)=\sum_{i\ge0}\frac{a^i b^{i+r}}{i!(i+r)!},
 \qquad T_r(a,b)=\sum_{j\ge r}C_j(a,b).
\]
Thus $p_{-r}=e^{-(a+b)}C_r$.  The mass sequence is log-concave on its
interval support.  Moreover $jp_j=ap_{j-1}-bp_{j+1}$, so at every positive
index in its support $p_j/p_{j-1}\le a/j<h$.  At negative indices,
termwise comparison gives $C_{r+1}\le b C_r/(r+1)$; hence the reverse
signed sequence $p_j-hp_{j-1}$ is strictly positive far enough to the left.
Its positive-part sum therefore has a last positive index $-r$, with
$r\ge0$, and telescoping gives
\begin{equation}
 \delta_{P\Vert P+1}(\eps)
 =e^{-(a+b)}\{C_r-(h-1)T_{r+1}\},\qquad C_r>hC_{r+1}.
 \label{eq:plateau-cutoff}
\end{equation}
Zero-payoff indices do not affect this identity.  The recurrence at
$-r-1$ and the same termwise comparison imply
\[
 \frac{C_r}{C_{r+1}}
 =\frac{r+1}{b}+\frac a b\frac{C_{r+2}}{C_{r+1}}
 \le\frac{r+1}{b}+\frac a{r+2}
 \le\frac{r+1}{b}+a.
\]
It follows that $r+1>b(h-a)\ge b(1+b)$.  Consequently
\begin{equation}
 \frac{b^r}{r!}\le1.
 \label{eq:plateau-factorial}
\end{equation}
Indeed, for $b\le1$ this is immediate.  For $b\ge1$ the preceding strict
inequality yields $r>b^2+b-1\ge b^2$, and
$(r!)^2=\prod_{i=1}^r i(r+1-i)\ge r^r$ gives
$b^r/r!\le(b/\sqrt r)^r\le1$.

For this fixed integer $r$ set $H_r(a,b)=C_r-(a+b)T_{r+1}$.
The absolutely convergent series give
$\partial_a C_r=C_{r+1}$ and $\partial_a T_{r+1}=T_{r+2}$, and thus
\[
 \partial_aH_r(a,b)=-(1+a+b)T_{r+2}(a,b)<0.
\]
Since $h-1\ge a+b$, the braces in \eqref{eq:plateau-cutoff} satisfy
\[
 C_r-(h-1)T_{r+1}\le H_r(a,b)\le H_r(0,b)
 =\frac{b^r}{r!}-b\sum_{j\ge r+1}\frac{b^j}{j!}\le1.
\]
If $a>0$ the middle inequality is strict.  The integer $r$ is kept fixed
while integrating the derivative from zero to $a$; it need not remain the
optimizing cutoff during this comparison.  Thus the reverse profile is
strictly below $e^{-\lambda}$ at every interior composition and at most
that value when $a=0$.  Reflection $j\mapsto1-j$ exchanges the directed
profiles and interchanges $a,b$, proving the same interior bound for the
other direction.  When $b=0$, the reverse profile is exactly $e^{-a}$,
since $h\ge1+a$ leaves only its atom at zero positive.  Its other direction
is bounded by $e^{-a}$ by the just-proved $a=0$ case.  Reflection treats
the other endpoint.  Hence both endpoints attain $e^{-\lambda}$ and every
interior composition is strictly below it, including $h=1+\lambda$.

Conversely, suppose $h<1+\lambda$.  For the fixed event $\{j\le0\}$ its
reverse payoff is
\[
 f(\theta)=\Prob(D_{\lambda,\theta}\le0)
               -h\Prob(D_{\lambda,\theta}\le-1).
\]
Expanding according to zero or one negative Poisson jump gives
\[
 \Prob(D_{\lambda,\theta}\le0)
 =e^{-\lambda}\{1+\lambda(1+\lambda)\theta\}+O_\lambda(\theta^2),
 \qquad
 \Prob(D_{\lambda,\theta}\le-1)
 =e^{-\lambda}\lambda\theta+O_\lambda(\theta^2).
\]
The omitted probability of at least two negative jumps is $O_\lambda(\theta^2)$,
and the zero- and one-jump terms are analytic in $\theta$.  Therefore
$f'(0)=e^{-\lambda}\lambda(1+\lambda-h)>0$.  Since
$f(0)=e^{-\lambda}$ and a fixed-event payoff lower-bounds its directed
profile, a sufficiently small interior composition has profile strictly
above the floor.  This proves the converse for every $1\le h<1+\lambda$.
\end{proof}

\begin{corollary}[An explicit minimax regime]
\label{cor:explicit-minimax}
If $e^\eps\ge1+\log(1/\delta)$, then
\begin{equation}
 \Rstar=\Lambda=\lammin=\log(1/\delta).
 \label{eq:explicit-minimax}
\end{equation}
\end{corollary}

\begin{proof}
Put $\lambda_0=\log(1/\delta)$.  Theorem~\ref{thm:floor-plateau} gives
$\Gamma_{\lambda_0}(\eps)=\delta$.  The endpoint profile has the same
value because its singular atom gives $D_\eps(\lambda_0)\ge
e^{-\lambda_0}=\delta$, while it is at most the global profile.  For every
$\lambda<\lambda_0$, that atom gives $D_\eps(\lambda)>\delta$.
The calibrations therefore coincide at $\lambda_0$, and
Theorem~\ref{thm:rr-classification} proves \eqref{eq:explicit-minimax}.
\end{proof}

\subsection*{A sufficient condition for an interior improvement}

\begin{proposition}[A boundary derivative test]
\label{prop:boundary-derivative-test}
Fix $\eps>0$, write $h=e^\eps$, and suppose
$\lambda/h=k+f$ with an integer $k\ge1$ and $0<f<1$.
Let $p_j=\Prob(\Poi(\lambda)=j)$, setting $p_j=0$ for $j<0$, and put
\[
 H_-(\lambda)=\sum_j(p_j-hp_{j-1})_+,
 \qquad S(\lambda)=f(k+1)-h(k+f)(1-f).
\]
If $H_-(\lambda)=D_\eps(\lambda)$ and $S(\lambda)>0$, then
\[
 \Gamma_\lambda(\eps)>D_\eps(\lambda).
\]
\end{proposition}

\begin{proof}
The sequence $g_j=p_j-hp_{j-1}=p_j(1-hj/\lambda)$ is positive
exactly at $j=0,\ldots,k$. Thus the fixed-event payoff
\[
 F_k(\theta)=\Prob(D_{\lambda,\theta}\le k)
                  -h\Prob(D_{\lambda,\theta}\le k-1)
\]
satisfies $F_k(0)=H_-(\lambda)$. It lower-bounds one directed profile
at every composition. Differentiating the Poisson jump laws at zero gives
\[
 \left.\partial_\theta\Prob(D_{\lambda,\theta}=j)\right|_{0+}
       =\lambda(p_{j+1}-p_{j-1}).
\]
Both jump intensities vary: differentiating $\lambda(1-\theta)$
contributes $-\lambda(p_{j-1}-p_j)$, and differentiating $\lambda\theta$
contributes $\lambda(p_{j+1}-p_j)$. Summing gives the corresponding
cumulative identity, with an $O_\lambda(\theta^2)$ remainder. Consequently,
\begin{align*}
 F_k'(0+)&=\lambda(g_k+g_{k+1})\\
 &=\frac{\lambda h p_{k-1}}{k(k+1)}
       \{f(k+1)-h(k+f)(1-f)\}>0.
\end{align*}
The last equality uses $p_k=(\lambda/k)p_{k-1}$ and
$p_{k+1}=(\lambda/(k+1))p_k$. A sufficiently small positive composition
therefore gives $\Gamma_\lambda(\eps)\ge F_k(\theta)>F_k(0)
=D_\eps(\lambda)$.
\end{proof}

\begin{corollary}[RR suboptimality at arbitrarily small targets]
\label{cor:small-target-gap}
For every fixed $\eps>0$, there is a strictly decreasing sequence
$\delta_m\downarrow0$ such that
\[
 \Lambda(\eps,\delta_m)<\lammin(\eps,\delta_m),\qquad
 \frac{\Lambda(\eps,\delta_m)+\lammin(\eps,\delta_m)}2
      <\Rstar(\eps,\delta_m)<\lammin(\eps,\delta_m).
\]
In particular, binary RR is strictly suboptimal at arbitrarily small
fixed privacy targets.
\end{corollary}

\begin{proof}
Put $h=e^\eps$ and choose
\[
 f_h=\frac{2h+1}{2(h+1)},\qquad
 \lambda_m=h(m+f_h),\qquad m=1,2,\ldots.
\]
Then $0<f_h<1$ and
\[
 S(\lambda_m)=\frac m2+
           f_h\left(1-\frac{h}{2(h+1)}\right)>0.
\]
The reverse endpoint direction $H_-$ is active for all sufficiently
large $\lambda$. Indeed, the directed upper bound in the proof of
Theorem~\ref{thm:global-tail}, extended to $\theta=0$ by continuity,
is $C_\eps\lambda^{-3/2}e^{-\lambda I(h)}$ for the other direction.
The single reverse summand used in the same proof gives
$H_-(\lambda)\ge c_\eps\lambda^{-3/2}e^{-\lambda I(h^{-1})}$
for all sufficiently large $\lambda$. Since $I(h)>I(h^{-1})$, the other
direction is strictly smaller eventually. Proposition~\ref{prop:boundary-derivative-test}
therefore gives $\Gamma_{\lambda_m}(\eps)>D_\eps(\lambda_m)$
after discarding finitely many initial indices.

Set $\delta_m=D_\eps(\lambda_m)$. Theorem~\ref{thm:calibration}
implies that these levels strictly decrease to zero and
$\Lambda(\eps,\delta_m)=\lambda_m$. The strict decrease of the global
profile implies $\lammin(\eps,\delta_m)>\lambda_m$.
Theorems~\ref{thm:rr-classification} and~\ref{thm:strict-midpoint}
give the remaining conclusions.
\end{proof}

The derivative test is sufficient; no converse is asserted. When
$S(\lambda)\le0$, this local argument does not rule out a more distant
interior maximum. The corollary fixes
each target before the population limit; it asserts neither separation
at every sufficiently small target nor a uniform perturbation size or
nonvanishing relative utility gain as $\delta_m\downarrow0$.

%% file: main.bbl
\begin{thebibliography}{99}
\setlength{\itemsep}{0pt}
\bibitem{EFM19}\'U.~Erlingsson, V.~Feldman, I.~Mironov, A.~Raghunathan, K.~Talwar, and A.~Thakurta. Amplification by shuffling: From local to central differential privacy via anonymity. In \emph{SODA}, 2019. \href{https://arxiv.org/abs/1811.12469}{arXiv:1811.12469}.

\bibitem{CSU19}A.~Cheu, A.~Smith, J.~Ullman, D.~Zeber, and M.~Zhilyaev. Distributed differential privacy via shuffling. In \emph{EUROCRYPT}, 2019. \href{https://arxiv.org/abs/1808.01394}{arXiv:1808.01394}.

\bibitem{BBGN19}B.~Balle, J.~Bell, A.~Gasc\'on, and K.~Nissim.
The privacy blanket of the shuffle model.
In \emph{Advances in Cryptology---CRYPTO 2019},
LNCS 11693, pages 638--667. Springer, 2019.
\href{https://doi.org/10.1007/978-3-030-26951-7_22}
{doi:10.1007/978-3-030-26951-7\_22}.
Author preprint: \href{https://arxiv.org/abs/1903.02837v2}
{arXiv:1903.02837v2}.

\bibitem{FMT21}V.~Feldman, A.~McMillan, and K.~Talwar. Hiding among the clones: A simple and nearly optimal analysis of privacy amplification by shuffling. In \emph{FOCS}, 2021. \href{https://arxiv.org/abs/2012.12803}{arXiv:2012.12803}.

\bibitem{KHH23}A.~Koskela, M.~A.~Heikkil\"a, and A.~Honkela. Numerical accounting in the shuffle model of differential privacy. \emph{Transactions on Machine Learning Research}, 2023. Preprint: \href{https://arxiv.org/abs/2106.00477}{arXiv:2106.00477}, titled \emph{Tight Accounting in the Shuffle Model of Differential Privacy}.

\bibitem{BJP22}S.~Biswas, K.~Jung, and C.~Palamidessi. Tight differential privacy guarantees for the shuffle model with $k$-randomized response. In \emph{Foundations and Practice of Security (FPS 2023)}, LNCS 14551, pages 440--458. Springer, 2024. \href{https://doi.org/10.1007/978-3-031-57537-2_27}{doi:10.1007/978-3-031-57537-2\_27}. Preprint: \href{https://arxiv.org/abs/2205.08858}{arXiv:2205.08858}.

\bibitem{Wang24}S.~Wang, Y.~Peng, J.~Li, Z.~Wen, Z.~Li, S.~Yu, D.~Wang, and W.~Yang. Privacy amplification via shuffling: Unified, simplified, and tightened. \emph{Proceedings of the VLDB Endowment}, 17(8):1870--1883, 2024. \href{https://arxiv.org/abs/2304.05007}{arXiv:2304.05007}.

\bibitem{SCW25}P.~Su, H.~Cheng, and P.~Wang.
Decomposition-based optimal bounds for privacy amplification via shuffling.
\href{https://arxiv.org/abs/2504.07414v6}{arXiv:2504.07414v6}, 2025.

\bibitem{VA18}F.~Valovich and F.~Ald\`a. Computational differential privacy from lattice-based cryptography. In \emph{Number-Theoretic Methods in Cryptology (NuTMiC 2017)}, LNCS 10737, pages 121--141. Springer, 2018. \href{https://arxiv.org/abs/1710.02036}{arXiv:1710.02036}.

\bibitem{AKL21}N.~Agarwal, P.~Kairouz, and Z.~Liu. The Skellam mechanism for differentially private federated learning. In \emph{Advances in Neural Information Processing Systems}, 34:5052--5064, 2021. \href{https://arxiv.org/abs/2110.04995}{arXiv:2110.04995}.

\bibitem{ShvGaussian}
A.~Shvets. Fixed-composition shuffle asymptotics in the full-support
Gaussian regime. \href{https://arxiv.org/abs/2602.09029v6}{arXiv:2602.09029v6}, 2026.

\bibitem{GKMP20}B.~Ghazi, R.~Kumar, P.~Manurangsi, and R.~Pagh.
Private counting from anonymous messages: Near-optimal accuracy with
vanishing communication overhead.
In \emph{Proceedings of the 37th International Conference on Machine
Learning}, PMLR 119, pages 3505--3514, 2020.
Corrected full version: \href{https://arxiv.org/abs/2106.04247v1}
{arXiv:2106.04247v1}, 2021.

\bibitem{TLMean26}S.~Takagi and S.~P.~Liew.
Shuffling-aware optimization for private vector mean estimation.
In \emph{Proceedings of the 43rd International Conference on Machine
Learning}, PMLR 306, pages 118280--118308, 2026.
\url{https://proceedings.mlr.press/v306/takagi26a.html}.
Author preprint: \href{https://arxiv.org/abs/2604.28032v1}
{arXiv:2604.28032v1}.

\bibitem{TLBlanket26}S.~Takagi and S.~P.~Liew.
Analysis of shuffling beyond pure local differential privacy.
\emph{Proceedings of the ACM on Management of Data},
4(2):1--26, 2026 (PODS 2026).
\href{https://doi.org/10.1145/3801892}{doi:10.1145/3801892}.
Author preprint: \href{https://arxiv.org/abs/2601.19154v4}
{arXiv:2601.19154v4}.

\bibitem{SCW26}P.~Su, H.~Cheng, and P.~Wang.
Which privacy blanket is optimal in the shuffle model?
\emph{Cryptology ePrint Archive}, Paper 2026/667, 2026.
\url{https://eprint.iacr.org/2026/667}.

\bibitem{GDD21}
A.~M.~Girgis, D.~Data, and S.~Diggavi.
R\'enyi differential privacy of the subsampled shuffle model in distributed
learning. In \emph{Advances in Neural Information Processing Systems},
34:29181--29192, 2021. \href{https://arxiv.org/abs/2107.08763v1}
{arXiv:2107.08763v1}.

\bibitem{BC20}
V.~Balcer and A.~Cheu. Separating local \& shuffled differential privacy
via histograms. In \emph{1st Conference on Information-Theoretic
Cryptography (ITC 2020)}, LIPIcs 163, pages 1:1--1:14, 2020.
\href{https://doi.org/10.4230/LIPIcs.ITC.2020.1}{doi:10.4230/LIPIcs.ITC.2020.1}.
Preprint: \href{https://arxiv.org/abs/1911.06879v4}{arXiv:1911.06879v4}.

\bibitem{War65}
S.~L.~Warner. Randomized response: A survey technique for eliminating
evasive answer bias. \emph{Journal of the American Statistical
Association}, 60(309):63--69, 1965.

\end{thebibliography}
